\documentclass[journal, final, 10pt, twocolumn]{IEEEtran}

\usepackage{amsmath, amssymb, amsthm}
\usepackage{graphicx, graphics, epsfig, color}
\usepackage{adjustbox}
\usepackage{subfigure}
\usepackage{tikz, pgfplots}
\usepackage{algorithm, algorithmic}
\usepackage{float}
\usepackage{multirow}
\usepackage{cuted, flushend}
\usepackage{midfloat}
\usepackage{bm}
\usepackage{fancyhdr}
\usepackage{enumerate}
\usepackage[numbers,square,sort&compress]{natbib}

\pgfplotsset{compat=1.4}

\newtheorem{theorem}{Theorem}

\newtheorem{lemma}{Lemma}
\newtheorem{remark}{Remark}
\newtheorem{corollary}{Corollary}
\newtheorem{definition}{Definition}
\newtheorem{assumption}{Assumption}

\newcommand{\toas}{\xrightarrow{\rm{a.s.}}}

\def\stud{{\rm S}}
\def\hub{{\rm H}}

\DeclareMathOperator{\tr}{{\rm tr}}

\def\c{c_n}

\def\C{\mathbf{C}_N}
\def\D{\mathbf{D}_N}
\def\E{\mathbf{E}_N}
\def\I{\mathbf{I}_N}

\def\Chalf{\mathbf{C}_N^{1/2}}
\def\Dhalf{\mathbf{D}_N^{1/2}}
\def\Chalfm{\mathbf{C}_N^{-1/2}}

\def\Chat{\mathbf{\hat C}_N}

\def\Chatrx{\mathbf{\hat C}_{(x_i)}}
\def\Chatra{\mathbf{\hat C}_{(a_i)}}

\def\Shat{\mathbf{\hat S}_N}
\def\ShatRnd{\mathbf{\hat S}_N^{\rm R}}

\def\An{{\mathbf{F}_N}}
\def\Ani{{\mathbf{F}_{N,(i)}}}
\def\AnI{{\mathbf{F}_N^{-1}}}
\def\AniI{{\mathbf{F}_{N,(i)}^{-1}}}
\def\Ex{{\rm E}_{\x_i}}

\def\Bni{{\mathbf{G}_{N,(i)}}}
\def\BniI{{\mathbf{G}_{N,(i)}^{-1}}}

\def\A{\mathbf{A}_N}
\def\B{\mathbf{B}_N}
\def\R{\mathbf{R}}
\def\H{\mathbf{H}}
\def\0{\mathbf{0}}
\def\T{\mathbf{T}}
\def\Q{\mathbf{Q}}
\def\Y{\mathbf{Y}}
\def\Z{\mathbf{Z}}

\def\x{\mathbf{x}}
\def\y{\mathbf{y}}
\def\a{\mathbf{a}}

\def\h{\mathbf{h}}

\def\gammaEq{{\gamma_n}}

\def\gammaEqRnd{{\gamma^{\rm R}_n}}
\def\alphaEqRnd{{\alpha^{\rm R}_n}}

\def\alphaEqi{{\alpha_{i,n}}}
\def\alphaEqj{{\alpha_{j,n}}}

\def\eps{\varepsilon_n}

\def\rnd{{\rm{R}}}
\def\deltaf{{\beta}}
\def\deltap{{\delta}}
\def\ind{{\mathbf{1}}}

\def\O{{\mathcal{O}}}

\title{Large Dimensional Analysis of Robust M-Estimators of Covariance with Outliers}

\author{David Morales-Jimenez$^\star$, Romain Couillet$^\dagger$, Matthew R. McKay$^\star$

\thanks{$^\star$D. Morales-Jimenez and M. R. McKay are with Dept. Electronic and Computer Engineering, Hong Kong University of Science and Technology, Clear Water Bay, Kowloon (Hong Kong). (e-mail:\{eedmorales,eemckay\}@ust.hk)}

\thanks{$^\dagger$R. Couillet is with CNRS-CentraleSup\'elec-Universit\'e Paris-Sud, 91192 Gif-sur-Yvette, France (romain.couillet@centralesupelec.fr). }

\thanks{The work of D. Morales-Jimenez and M. R. McKay was supported by the Hong Kong Research Grants Council under grant number 16206914. Couillet's work is supported by the ERC MORE EC--120133.}

\thanks{This work has been submitted to the IEEE for possible publication. Copyright may be transferred without notice, after which this version may no longer be accessible}
}

\begin{document}

\date{\today}

\maketitle

\begin{abstract}
	A large dimensional characterization of robust \mbox{M-estimators} of covariance (or scatter) is provided under the assumption that the dataset comprises independent (essentially Gaussian) legitimate samples as well as arbitrary deterministic samples, referred to as outliers. Building upon recent random matrix advances in the area of robust statistics, we specifically show that the so-called Maronna M-estimator of scatter asymptotically behaves similar to well-known random matrices when the population and sample sizes grow together to infinity. The introduction of outliers leads the robust estimator to behave asymptotically as the weighted sum of the sample outer products, with a constant weight for all legitimate samples and different weights for the outliers. A fine analysis of this structure reveals importantly that the propensity of the M-estimator to attenuate (or enhance) the impact of outliers is mostly dictated by the alignment of the outliers with the inverse population covariance matrix of the legitimate samples. Thus, robust M-estimators can bring substantial benefits over more simplistic estimators such as the per-sample normalized version of the sample covariance matrix, which is not capable of differentiating the outlying samples.
The analysis shows that, within the class of Maronna's estimators of scatter, the Huber estimator is most favorable for rejecting outliers. On the contrary, estimators more similar to Tyler's scale invariant estimator (often preferred in the literature) run the risk of inadvertently enhancing some outliers.
\end{abstract}

\section{Introduction}
\label{sec:intro}

The growing momentum of big data applications along with the recent advances in large dimensional random matrix theory have raised much interest for problems in statistics and signal processing under the assumption of large but similar population dimension $N$ and sample size $n$. 
Due to the intrinsic complexity of large dimensional random matrix theory, as compared to classical statistics where $N$ is fixed and $n\to\infty$, most of the classical applications were concerned with sample covariance matrix (SCM) based methods (as in e.g., \cite{Bianchi2009,Nadler2010} for source detection or \cite{Mestre2008} for subspace estimation). Only recently have other random matrix structures started to be explored which are adequate to deal with more advanced statistical problems; see for instance \cite{BIC08} on Toeplitz random matrix structures, or \cite{ELK10} on kernel random matrices. Of particular interest is the structure of robust M-estimators of covariance (or scatter), which have very recently come to a better understanding in the large dimensional regime and is the focus of the present work. 


The field of robust M-estimation, born with the early works of Huber \cite{Huber1964}, roughly consists in improving classical Gaussian maximum-likelihood estimators, such as the sample mean or SCM, into estimators that (unlike the classical estimators) are resilient to both the possibly heavy-tailed nature of the observed data or the presence of outliers in the dataset. Assuming observation data of known zero mean, robust estimators of the population covariance matrix, referred to as robust \mbox{M-estimators} of scatter, were proposed successively in \cite{Huber1964} for data composed of a majority of independent Gaussian samples and a few outliers and then in \cite{Maronna1976} and \cite{Tyler1987} for elliptically distributed or arbitrary scaled Gaussian data.

But the analysis for each given $N,n$ of the aforementioned robust estimators of scatter, which often take the form of solutions of implicit equations, is in general intractable. In a series of recent works \cite{Couillet2013,COU14,COU14d,ZHA14} (see also \cite{YAN14,COU14c} for applications), this limitation was alleviated by considering the random matrix regime where both $N,n$ are large and commensurable. These works have shown that in this regime several classes of robust estimators of scatter (Maronna, Tyler, and regularized Tyler) behave similar to simpler and explicit random matrix models, which are fully understandable via (now standard) random matrix methods. Nonetheless, all these works were pursued under the assumption that the input data are independent and follow a zero-mean elliptical distribution. One of the salient outcomes of these works is that, under elliptical inputs, the Tyler and regularized Tyler estimators asymptotically behave similar to the SCM of the normalized data,\footnote{This being valid up to second-order fluctuations \cite{COU14d}.} henceforth referred to as the normalized SCM, and therefore do not provide any apparent gain in robustness versus simpler sample covariance estimators. 

This fact, however, fundamentally disregards the important role of robust estimators as arbitrary outlier rejectors. In the present work, 
we shall consider data comprising both legitimate data (that are essentially independent Gaussian samples) and a certain (a priori unknown) amount of arbitrary deterministic outliers. Focusing our attention specifically to the (larger) class of Maronna's M-estimators of scatter, similar to all of the aforementioned works and following the approach in \cite{Couillet2013}, we will show that in this setting the robust estimator of scatter behaves similar for large $N,n$ to an explicit and easily understood random matrix. But it will appear, unlike in \cite{Couillet2013,COU14,COU14d,ZHA14}, that this random matrix no longer behaves similar to the normalized SCM. Our main finding is that, under suitable conditions, the robust estimator of scatter manages to attenuate (to some extent) the impact of the deterministic outliers, which the SCM (or normalized SCM) may not be capable of. Calling $\C$ the population covariance matrix of the legitimate data, $\a_i\in\mathbb{C}^N$ the $i$-th outlier, and assuming the number of outliers is small compared to $n$, it will be demonstrated that the rejection power of the robust estimator of scatter is monotonically related to the quadratic form $\a_i^\dagger\C^{-1}\a_i$. This shows that, if $\C$ is (invertible but) essentially of low rank, $\a_i^\dagger\C^{-1}\a_i$ can take large values and thus $\a_i$ is likely to be suppressed. If $\a_i^\dagger\C^{-1}\a_i$ is quite small instead, an inverse effect of outlier enhancement may appear that needs be controlled by an appropriate choice of estimator within Maronna's class. We shall show that such an estimator should resemble the original Huber estimator from \cite{Huber1964} and substantially differ from the Tyler estimator.

In the remainder of the article, we provide a rigorous statement of our main results. The problem at hand is discussed in Section~\ref{sec:problem} and our main results introduced in Section~\ref{sec:main}, all proofs being deferred to the appendices. Special attention will then be made on the analytically tractable cases where the number of outliers is either small (Section~\ref{sec:few_outliers}) or random i.i.d.\@ (Section~\ref{sec:random}). Concluding remarks are provided in Section~\ref{sec:conclusion}.    

\medskip

\textit{Notations:} The superscript $(\cdot)^\dagger$ stands for Hermitian transpose in the complex case or transpose in the real case. The norm $\|\cdot\|$ is the spectral norm for matrices and the Euclidean norm for vectors. The Dirac measure at point $x$ is denoted $\delta_x$ and $\ind_{A}$ stands for the indicator function with $A$ the corresponding inclusion event. The imaginary unit is denoted $\imath = \sqrt{-1}$ and $\Im [ \cdot ]$ stands for the imaginary part. The set $\mathbb{R}^+$ is defined as $\{x: x \geq 0 \}$ and $\mathbb{C}^+=\{z\in\mathbb{C},~\Im[z]>0\}$. The support of a distribution function $F$ is denoted by ${\rm{Supp}}(F)$. The ordered eigenvalues of a Hermitian (or symmetric) matrix ${\bf X}$ of size $N \times N$ are denoted $\lambda_1 ({\bf X}) \leq \ldots \leq \lambda_N ({\bf X})$. For $\mathbf{A}, \mathbf{B}$ Hermitian, $\mathbf{A} \succ \mathbf{B}$ means that $\mathbf{A} - \mathbf{B}$ is positive definite. The notation ${\rm diag}({\bf X})$ stands for the diagonal matrix composed of the diagonal elements of matrix $\bf X$ and ${\rm diag}({\bf x})$ the diagonal matrix composed of the elements of vector $\bf x$ on the diagonal. The arrow $\toas$ designates almost sure convergence and $\Rightarrow$ stands for weak convergence.

\section{System Model and Motivation}
\label{sec:problem}

For $\eps\in \mathbb{R}$ such that $n\eps\in \{1,\ldots,n\}$, let
\begin{align*}
	\Y = \left[ \y_1, \ldots, \y_{(1-\varepsilon_n)n}, \a_1, \ldots , \a_{\varepsilon_n  n} \right] \in \mathbb{C}^{N\times n}
\end{align*}
where $\y_i= \Chalf \x_i\in \mathbb{C}^{N}$, $i=1,\ldots,(1-\varepsilon_n)n$, are independent across $i$, $\C \in \mathbb{C}^{N \times N}$ is deterministic Hermitian positive definite, and $\x_i$ has zero mean, unit variance and finite $(8+\eta)$-th order moment entries for some $\eta>0$, while $\a_1,\ldots,\a_{\varepsilon_n n} \in \mathbb{C}^{N}$ are arbitrary deterministic vectors.\footnote{As shall be seen in Section~\ref{sec:random}, the vectors $\a_i$'s can be considered random as long as they are independent of the $\y_i$'s.}
We shall further assume that, as $N\to\infty$, $\limsup_N \| \C \| < \infty$. 

The vectors $\y_1, \ldots, \y_{(1-\varepsilon_n)n}$ will be considered the legitimate data, while $\a_1,\ldots,\a_{\varepsilon_n n}$ are deterministic unknown outliers.
It is important to note at this point that all estimators of $\C$ considered in the following are invariant to column permutations in $\Y$ so that we can freely assume the first columns of $\Y$ to be the legitimate data and the last columns to be the outliers.
Note also that we consider here a more general setting than Gaussian legitimate data as we merely request the $\x_i$'s to have independent normalized entries with some bounded moment condition. 

Although $\a_1,\ldots,\a_{\varepsilon_n n}$ are arbitrary, for technical reasons we shall need the following control. 
\begin{assumption}
	\label{as:C}
	$\limsup_n \| \frac1n \sum_{i=1}^{\eps n} \Chalfm \a_i \a_i^\dagger \Chalfm \|< \infty .$
\end{assumption}
Note that, if $\limsup_n \eps n < \infty$, Assumption \ref{as:C} reduces to $\limsup_n \max_{1\leq i \leq \eps n} \frac1N\a_i^*\C^{-1}\a_i < \infty$.

\bigskip

If one were aware of the presence and position of outliers in the dataset, then the natural estimator for $\C$ (up to renormalization by $1-\varepsilon_n$) would read $\frac1n\Y^{\rm o}{\Y^{\rm o}}^\dagger$ with $\Y^{\rm o}=[\y_1,\ldots,\y_{(1-\varepsilon_n)n}]$; this estimator, which we shall refer to as the Oracle estimator (hence the ``${\rm o}$'' superscript), merely consists in a SCM with discarded outliers. 
For lack of knowing the outliers presence and positions, the immediate alternative estimate for $\C$ is the SCM, which reads here $\frac1n\Y\Y^\dagger$. If one is only interested in estimating any scaled version of $\C$, then, to mitigate the negative impact of outliers with arbitrarily large norm, a simple robust procedure consists in estimating $\C$ via the normalized SCM $\frac1n{\Y^{\rm n}}{\Y^{\rm n}}^\dagger$, where ${\Y^{\rm n}}\triangleq \Y{\rm diag}(\frac1N\Y^\dagger\Y)^{-\frac12}$. This matrix has the advantage of avoiding arbitrarily large biases in the estimation of $\C$. However, being only based on a per-data norm control, $\frac1n{\Y^{\rm n}}{\Y^{\rm n}}^\dagger$ does not take into account the fact that outliers can also be detected if they significantly differ, not just in norm, from the majority of the data. The robust estimators of scatter, introduced by Huber \cite{Huber1964} and later studied by Maronna \cite{Maronna1976}, were precisely designed for this purpose. Our objective here is to finely understand this outlier identification and mitigation procedure by means of a large random matrix analysis.

To be able to define a robust M-estimator of scatter in the sense of Maronna under the presence of arbitrary outlier vectors, a constraint must be set on $\varepsilon_n$ and $N$. In particular, as $n$ grows large, we shall require that $n(1-\varepsilon_n)/N$ (and not only $n/N$) be always beyond one.
\begin{assumption}[Growth rate]
\label{as:cN}
As $n \to \infty$ $\varepsilon_n \to \varepsilon\in [0,1)$ and $\c \triangleq \frac{N}n \to c$ with $0 < c < 1-\varepsilon$.
\end{assumption}

We then define Maronna's $M$-estimator of scatter $\Chat$ as a solution, when it exists, to the equation in $\Z$
\begin{align} \label{Maronna}
\Z &= \frac{1}{n} \sum_{i=1}^{(1-\varepsilon_n)n} u\left( \frac{1}{N} \y_i^\dagger \Z^{-1}  \y_i \right) \y_i \y_i^\dagger \nonumber \\
   &+ \frac{1}{n} \sum_{i=1}^{\varepsilon_n n} u\left( \frac{1}{N} \a_i^\dagger \Z^{-1} \a_i  \right) \a_i \a_i^\dagger.
\end{align}
where $u:[0,\infty)\to (0,\infty)$ is continuous, non-increasing, and such that $\phi(x) \triangleq x u(x)$ is increasing with $\lim_{x\to \infty} \phi (x) \triangleq \phi_{\infty}$ and $(1-\varepsilon)^{-1}<\phi_{\infty}<c^{-1}$. 
	Note that the latter assumption on $\phi_{\infty}$ is equivalent to that in \cite{Couillet2013} with a slight modification accounting for the presence of outliers.
	
	A standard choice for the function $u$ is $u=u_\stud$, where, for some $t>0$,
\begin{align}
	\label{eq:u_mar}
	u_\stud(x)\triangleq \frac{1+t}{t+x}
\end{align}
which, for an appropriate $t$, turns $\Chat$ into the maximum-likelihood estimator of $\C$ when the columns of $\Y$ are independent multivariate Student vectors (hence the superscript ``$\stud$'').
As $t\to 0$, $\Chat$ converges to one version of the so-called Tyler estimator \cite{Tyler1987}, as shown in \cite{CHI14}.\footnote{As opposed to Maronna's class of estimators, Tyler estimator is only defined up to a constant factor; thus it estimates $\C$ up to a scale parameter.} We shall however restrict our study here to Maronna's class of estimators. Of particular interest in the present work is another function $u$, which we shall (somewhat abusively\footnote{Huber's original estimator takes the form $u(x)=\max\{\alpha,\beta/x\}$ for some $\alpha,\beta$, hence with additional parameters and with $t=0$. However, uniqueness of $\Chat$ is not guaranteed for $t=0$ and, in the random matrix limit, $\alpha=\beta=1$ is a particularly appealing choice.}) refer to as Huber's estimator function $u_\hub$, defined, for some $t>0$, as
\begin{align}
	\label{eq:u_hub}
	u_\hub(x)\triangleq \max\left\{ 1, \frac{1+t}{t+x}\right\}.
\end{align}
This function has the particularity of being constant for all $x\leq 1$, which will be later seen as an important property. 

\section{Main Result}
\label{sec:main}

From the problem setting, Assumption~\ref{as:cN}, and \cite[Thm.~2.3]{Kent1991}, it is easily seen that, with probability one, the solution of \eqref{Maronna} is unique for all large $n$ and thus $\Chat$ is unequivocally defined. In the same spirit as in \cite{Couillet2013,COU14} (and with similar notations), our first objective is to find an explicit tight approximation of the implicitly defined $\Chat$ in the regime where $N,n\to\infty$ as per Assumption~\ref{as:cN}. Our main result unfolds as follows.

\begin{theorem}[Asymptotic Behavior] 
\label{thm1}
Let Assumptions~\ref{as:C}--\ref{as:cN} hold and let $\Chat$ the solution to \eqref{Maronna} (unique for all large $n$, with probability one). Then, as $n \to \infty$,
\begin{align*}
\left\| \Chat - \Shat \right\| \toas 0 
\end{align*}
where
\begin{align*}
 \Shat & \triangleq v\left( \gammaEq \right) \frac{1}{n} \sum_{i=1}^{(1-\varepsilon_n)n}  \y_i \y_i^\dagger + \frac{1}{n} \sum_{i=1}^{\varepsilon_n n} v\left( \alphaEqi \right) \a_i \a_i^\dagger 
\end{align*}
with $v(x) = u \left( g^{-1}(x) \right)$, $g(x) = x/(1-c\phi (x))$, and $(\gammaEq,\alpha_{1,n},\ldots,\alpha_{\varepsilon_n n,n})$ the solution to 
\begin{align}
\gammaEq &= \frac{1}{N} \tr \C \hspace{-1mm} \left(\frac{(1-\varepsilon) v(\gammaEq)}{1+c v(\gammaEq) \gammaEq} \C + \frac{1}{n} \sum_{j=1}^{\varepsilon_n n} v ( \alphaEqj )  \a_j \a_j^\dagger  \right)^{\hspace{-1.5mm} -1} \nonumber \\ 
\alphaEqi&= \frac{1}{N} \a_i^\dagger \left( \frac{(1-\varepsilon) v(\gammaEq)}{1+c v(\gammaEq) \gammaEq} \C + \frac{1}{n} \sum_{j \neq i} v( \alphaEqj )  \a_j  \a_j^\dagger \right)^{\hspace{-1.5mm} -1} \hspace{-2mm} \a_i \label{gammaThm}
\end{align}
for $i=1,\ldots,\eps n$.
In particular, from \cite[Thm.~4.3.7]{Horn1985},
\begin{align*}
\max_{1\leq i \leq N} \left| \lambda_i(\Chat)-\lambda_i(\Shat)\right|\toas 0.
\end{align*}
\end{theorem}

\begin{remark}[Function $v$]
	The function $v$ defined in Theorem~\ref{thm1} was already introduced in \cite{Couillet2013} and uses, through $g$, the assumption that $\phi(x)<c^{-1}$. It has essentially the same general properties as $u$ in that it is continuous, non-increasing and such that $\psi(x) \triangleq x v(x)$ is increasing and bounded with $\lim_{x\to \infty} \psi (x) \triangleq \psi_{\infty} = \phi_{\infty} / (1-c \phi_{\infty})$. 
\end{remark}

\begin{remark}[Relation to previous results]
	\label{rem:previous_results}
	Taking $\varepsilon_n=0$, Theorem~\ref{thm1} reduces to the result obtained in \cite{Couillet2014a} and \cite{ZHA14}, i.e., $\Shat=v(\gammaEq) \frac1n\Y\Y^\dagger$. In this case, \eqref{gammaThm} reduces to
	\begin{align*}
	 \gammaEq &= \frac{1+c v(\gammaEq) \gammaEq}{v(\gammaEq)}
	\end{align*}
	which, after basic algebra, entails $\gammaEq=\phi^{-1}(1)/(1-c)$ and $v(\gammaEq)=1/\phi^{-1}(1)$.
\end{remark}

Theorem~\ref{thm1} allows us to transfer many properties of the implicit matrix $\Chat$ into the more tractable matrix $\Shat$, the random matrix structure of which is well known and has been studied as early as in \cite{Silverstein1995a}. The structure of $\Shat$ is particularly interesting as it mostly consists of two terms: the sum of outer products of the legitimate data scaled by a constant factor $v(\gamma_n)$ along with a per-sample weighted sum of the outer products of the outlying data. Therefore, as one would expect, $\Chat$ sets a specific emphasis (either small or large) on each outlying sample while maintaining all legitimate data under constant weight. We expect here that, as opposed to the SCM that provides no control on the data or to the normalized SCM that merely normalizes the outliers, $\Chat$ will appropriately ensure a reduction of the outlier impact by letting $v(\alpha_{j,n})$ be quite small compared with $v(\gamma_n)$, especially if $\eps$ is small.

An immediate corollary of Theorem~\ref{thm1} concerns the large $N$ eigenvalue distribution of $\Chat$ and reads as follows.

\begin{corollary}[Spectral Distribution] Define the empirical spectral distribution $F_N^{\Chat}(x) = \frac1N \sum_{i=1}^N \ind_{\{\lambda_i (\Chat) \leq x\}}$ for $x \in \mathbb{R}$. Then, under the setting of Theorem~\ref{thm1},
\begin{align}
F_N^{\Chat} (x) - F_N (x) \Rightarrow 0
\notag
\end{align}
almost surely as $n \to \infty$, where $F_N (x)$ is a real distribution function with density defined via its Stieltjes transform $m_N(z)$ (i.e.,\footnote{Recall that any distribution function $F$ is uniquely defined by its Stieltjes transform $m(z)$ by the fact that, for all continuity points $a,b$ of $F$, $$F(b)-F(a)=\lim_{y\downarrow 0} \int_a^b \Im [m(t+\imath y)]dt.$$} $m_N(z)\triangleq \int (t-z)^{-1}dF_N(t)$) given for all $z \in \mathbb{C}^+$ by
\begin{align*}
m_N(z) = \frac1N\tr \left( \frac{(1-\varepsilon) v(\gammaEq)}{1+e_{N}(z)} \C + \A - z \I \right)^{-1}
\end{align*}
with $\A \triangleq \frac{1}{n} \sum_{i=1}^{\varepsilon_n n} v\left( \alphaEqi \right) \a_i \a_i^\dagger$ and $e_N(z)$ the unique solution in $\mathbb{C}^+$ of the equation
\begin{align*}
	e_{N}(z) = \frac{v(\gammaEq)}n\tr\C \left( \frac{(1-\varepsilon) v(\gammaEq)}{1+e_{N}(z)} \C + \A - z \I \right)^{-1}.
\end{align*}
\label{cor2}
\end{corollary}


In the appendix, it is importantly shown that $\limsup_N\|\Chat\|<\infty$ a.s. (as a result of $\limsup_N\|\Shat\|<\infty$ a.s.). This implies that $F_N^{\Chat}$ and $F_N$ have compact supports and are fully determined by their respective moments $M^{\Chat}_{N,k}\triangleq \int t^k dF_N^{\Chat}(t)$ and $M_{N,k}\triangleq \int t^k dF_N(t)$, $k=1,2,\ldots$, which satisfy $M^{\Chat}_{N,k}-M_{N,k}\toas 0$ (by the dominated convergence theorem). While $F_N$ is defined via its deterministic but implicit Stieltjes transform, the $M_{N,k}$ can be retrieved explicitly using successive derivatives of the moment generating formula (for $|z|<1/\sup({\rm Supp}(F_N))$)
\begin{align*}
	m_N(1/z) &= - \sum_{k=0}^\infty z^{k+1} M_{N,k}.
\end{align*}
Precisely, we obtain here the following result.
\begin{corollary}[Moments]
	For $F_N$ defined in Corollary~\ref{cor2}, letting $M_{N,p}\triangleq \int t^p dF_N(t)$, $p=1,2,\ldots$, 
\begin{align}
M_{N,p} = \frac{(-1)^p}{p!} \frac1N \tr \T_p
\notag
\end{align}
where $\T_p$ is obtained from the following recursive formulas
\begin{align*}
\T_{p+1} &= - \hspace{-0.5mm} \sum_{i=0}^p \T_{p-i} \A \T_i \hspace{-0.3mm} + \hspace{-0.3mm} \sum_{i=0}^p \sum_{j=0}^i \binom{p}{i} \binom{i}{j} \T_{p - i} \Q_{i \hspace{-0.2mm}  - \hspace{-0.2mm} j \hspace{-0.2mm} + \hspace{-0.2mm} 1} \T_j \notag \\
\Q_{p+1} &= (p+1) f_{p} (1-\varepsilon) v(\gammaEq) \C \notag \\
f_{p+1} &= \sum_{i=0}^p \sum_{j=0}^i \binom{p}{i} \binom{i}{j} (p-i+1) f_{j} f_{i-j} \deltaf_{p-i} \notag \\
\deltaf_{p+1} &= v(\gammaEq) \frac1n \tr \C \T_{p+1} ,
\end{align*}
with initial values $\T_0 = \I$, $f_{0} = -1$, $\deltaf_{0} = v(\gammaEq) \frac1n \tr \C$.
In particular,
\begin{align*}
 M_{N,1} &= \frac1N \tr \left[ \A + (1-\varepsilon) v(\gammaEq) \C \right] \\
 M_{N,2} &= \frac1N \tr \Big[ \A^2 + 2 (1-\varepsilon) v(\gammaEq) \C \A \nonumber \\
	 & + (1-\varepsilon)^2 v^2(\gammaEq) \C^2 + \left[\frac1n\tr \C \right] (1-\varepsilon) v^2(\gammaEq) \C  \Big].
\end{align*}
	\label{cor:moments}
\end{corollary}

Albeit having characterized the random matrix $\Shat$, which approximates  the behavior of $\Chat$ for large $N,n$, it is quite challenging to gain a good intuitive understanding of the weight structure as the expression \eqref{gammaThm} relating $\gammaEq$ to the $\alphaEqi$'s is still implicit (while being deterministic). To get more insight on the properties of $\Chat$, we shall successively consider two specific scenarios that simplify the system \eqref{gammaThm}.

\section{Finitely Many Outliers Scenario}
\label{sec:few_outliers}

Let us first assume that $\eps n=K$ is maintained constant as $n\to\infty$ (thus $\varepsilon=0$). Recall that, in this scenario, Assumption \ref{as:C} can be replaced by the sufficient condition $\limsup_N \max_{1\leq i \leq \eps n} \frac1N\a_i^*\C^{-1}\a_i < \infty$. In the appendix, it is shown that $\gamma_n$ cannot grow unbounded with $n$. As such, by a rank-one perturbation argument iterated $K$ times, see e.g., \cite[Lemma 2.6]{Silverstein1995a}, we find that 
\begin{align}
	\gammaEq - \frac{1+c v(\gammaEq) \gammaEq}{v(\gammaEq)} = \O(1/N)
	\notag
\end{align}
which ensures by Remark~\ref{rem:previous_results} that 
\begin{align}
	\gamma_n = \frac{\phi^{-1}(1)}{1-c} + \O(1/N).
	\notag
\end{align}
We shall denote next $\gamma\triangleq \frac{\phi^{-1}(1)}{1-c}$ (and thus $v(\gamma)=1/\phi^{-1}(1)$).
Then we obtain that
\begin{align*}
	\left\Vert \Chat- \Shat' \right\Vert \toas 0
\end{align*}
with 
\begin{align*}
	\Shat' = v\left( \gamma \right) \frac1n\sum_{i=1}^{n-K} \y_i\y_i^\dagger + \frac1n \sum_{i=1}^K v(\alpha'_{i,n}) \a_i\a_i^\dagger
\end{align*}
where $\alpha'_{i,n}$ are the unique positive solutions to
\begin{align*}
	\alpha'_{i,n} &= \frac{1}{N} \a_i^\dagger \left( \gamma^{-1} \C + \frac{1}{n} \sum_{j \neq i} v( \alpha'_{j,n} )  \a_j  \a_j^\dagger \right)^{-1} \a_i.
\end{align*}
As such, when the number $K$ of outliers is fixed, the common weight $v(\gamma_n)$ becomes independent of the vectors $\a_i$'s (even if they are of arbitrarily large norm) while the individual weights $v(\alpha_{i,n})$ eventually solve a system of $K$ equations involving the $\a_i$'s and $\C$.

A more specific case lies in the scenario where $\a_1=\ldots=\a_K$. There, $\alpha_{1,n}=\ldots=\alpha_{K,n}$ and the $K$ equations above reduce to a single one reading
\begin{align*}
	\alpha'_{1,n} &= \frac{1}{N} \a_1^\dagger \left( \gamma^{-1} \C + \frac{K-1}{n} v( \alpha'_{1,n} )  \a_1  \a_1^\dagger \right)^{-1} \a_1
\end{align*}
which, using $\a^\dagger(\mathbf{A} + t\a\a^\dagger)^{-1}=\a^\dagger\mathbf{A}^{-1}/(1+t\a^\dagger \mathbf{A}^{-1} \a)$ for invertible $\mathbf{A}$, simplifies as
\begin{align*}
	\alpha'_{1,n} &= \gamma \frac{ \frac1N \a_1^\dagger \C^{-1}\a_1 }{ 1 +  \c \gamma (K-1) v( \alpha'_{1,n} ) \frac1N \a_1^\dagger \C^{-1} \a_1 }
\end{align*}
or equivalently
\begin{align*}
	\frac{\alpha'_{1,n}}{1-c_n(K-1)\psi(\alpha'_{1,n})} &= \gamma \frac1N \a_1^\dagger \C^{-1}\a_1.
\end{align*}
Since the right-hand side is positive, so should be the left-hand side, which may then be seen as an increasing function of $\alpha'_{1,n}$. Thus, since $\gamma$ depends neither on $\C$ nor $\a_1$, it comes that $\alpha'_{1,n}$ is an increasing function of $\frac1N \a_1^\dagger \C^{-1}\a_1$. Moreover, $\alpha'_{1,n}<\psi^{-1}(1/(\c(K-1)))$ and thus converges to zero as $K$ grows large. 
When $K=1$, and thus the outlier is now isolated, this reduces to
\begin{align*}
	\alpha'_{1,n} &= \gamma \frac1N \a_1^\dagger \C^{-1}\a_1.
\end{align*}

This short calculus leads to two important remarks. First, for $K=1$, $\Chat$ asymptotically allocates a weight $v(\gamma)$ to the legitimate data and a weight $v(\gamma \frac1N \a_1^\dagger \C^{-1}\a_1)$ for the outlier. As a consequence, by the non-increasing property of $v$, the effect of the outlier will be (for most choices of the $v$ function) attenuated if $\frac1N \a_1^\dagger \C^{-1}\a_1\geq 1$ but will be increased if $\frac1N \a_1^\dagger \C^{-1}\a_1\leq 1$. As such, the robust estimator of scatter will tend to mitigate the effect of outliers $\a_1$ having either large norm or, more interestingly, having strong alignment to the weakest eigenmodes of $\C$. In particular, note that when $\C={\bf I}_N$, $\Chat$ will mostly control outliers upon their norms $\frac1N \Vert\a_1\Vert^2$, which is essentially what the normalized SCM
\begin{align}
	\label{eq:normalized_SCM}
	\frac1n{\Y^{\rm n}}{\Y^{\rm n}}^\dagger = \frac1n\sum_{i=1}^{(1-\varepsilon_n)n} \frac{\y_i\y_i^\dagger}{\frac1N\Vert \y_i\Vert^2}+\frac1n\sum_{i=1}^{\varepsilon_nn} \frac{\a_i\a_i^\dagger}{\frac1N\Vert\a_i\Vert^2}
\end{align}
would do, and thus there is no gain in using robust estimators here. However, if $\C$ has large dimensional weak eigenspaces (i.e., close to singular with most eigenvalues near zero), $\frac1N \a_1^\dagger \C^{-1}\a_1$ may be quite large, and thus $\a_1$ may be strongly attenuated. But if $\a_1$ aligns to the strong eigenmodes of $\C$, the impact of $\a_1$ may be enhanced rather than reduced. To avoid this effect, undesirable in most cases, it is crucial to appropriately choose the $u$ function. Specifically, the function $v$ should be taken constant for all $x\leq \gamma$, or equivalently, $u(x)$ should be taken constant for $x\leq \phi^{-1}(1)$. A natural choice is the Huber estimator $u=u_\hub$ introduced in \eqref{eq:u_hub}.

The second remark is a slightly more surprising outcome. Indeed, despite $n$ being potentially extremely large, the presence of (already few) $K>1$ identical outliers drives $\Shat$ (and thus $\Chat$) to allocate large weights $v(\alpha_{i,n})$ (since $\alpha_{i,n}$ is small) to these outliers, therefore seemingly contradicting the very purpose of the robust estimator. This seems to indicate that $\Chat$ has the propensity to put forward both large quantities of data with similar distribution \emph{as well as} rather small quantities of vectors with strong pairwise alignment, while more naturally rejecting isolated outliers.

In terms of large dimensional spectral distribution and moments, the scenario of finitely many outliers is asymptotically equivalent to the outlier-free scenario. This can be observed from a rank-one perturbation argument along with $\eps\to 0$ applied to Corollaries~\ref{cor2}--\ref{cor:moments}.
A similar reasoning would hold for the normalized SCM. However, the matrices $\Chat$ and $\frac1n{\Y^{\rm n}}{\Y^{\rm n}}^\dagger$ themselves experience a (maximum) rank-$K$ perturbation which can severely compromise the estimation of $\C$, along the previous argumentation lines.

Figure~\ref{fig:oneoutlier_eigenvalues} displays an artificially generated scenario where a single outlier $\a_1$ of norm $\frac1N\Vert\a_1\Vert^2=1$ produces a large value for $\frac1N\a_1^\dagger\C^{-1}\a_1$ ($=14.50$), thus entailing a strong attenuation by $\Chat$. The terms $\a_1$ and $\C$ were made such that the SCM and normalized SCM have the same asymptotic eigenvalues and produce an isolated eigenvalue (around $.25$). The spectra of the latter are compared against those of $\Chat$ and the oracle estimator. It is seen that the isolated eigenvalue, which is naturally not present in the spectrum of the oracle estimator, is also not present in the spectrum of $\Chat$, indicating that $\Chat$ has significantly reduced its impact on the spectrum.

\begin{figure}[t!]
  \centering
  \begin{tikzpicture}[font=\footnotesize]
    \renewcommand{\axisdefaulttryminticks}{4} 
    \tikzstyle{every major grid}+=[style=densely dashed]       
    \tikzstyle{every axis y label}+=[yshift=-10pt, anchor=near ticklabel] 
    \tikzstyle{every axis x label}+=[yshift=5pt]
    \tikzstyle{every axis legend}+=[cells={anchor=west},fill=white,
        at={(0.98,0.98)}, anchor=north east, font=\scriptsize ]
    /xlabel near ticks
    /ylabel near ticks
    \begin{axis}[
      xmin=0,
      ymin=0,
      xmax=5,
      xtick={1,2,3,4},
      xticklabels={ $\frac1n\Y\Y^\dagger$ , $\frac1n\Y^{\rm n}{\Y^{\rm n}}^\dagger$ , $\Chat$ , $\frac1n\Y^{\rm o}{\Y^{\rm o}}^\dagger$ },
      bar width=1.5pt,
      grid=major,
      scaled ticks=true,
      ]
      \addplot[black,smooth,mark=x,only marks,line width=0.5pt] plot coordinates{
	      (1,2.180285)(1,0.042225)(1,0.046723)(1,0.048491)(1,0.051460)(1,0.054643)(1,0.057176)(1,0.061348)(1,0.064130)(1,0.067557)(1,0.250559)(1,0.374986)(1,0.406719)(1,0.421510)(1,0.438070)(1,0.466887)(1,0.471776)(1,0.487778)(1,0.501986)(1,0.516132)(1,0.531284)(1,0.542032)(1,0.554612)(1,0.562525)(1,0.570804)(1,0.594252)(1,0.606292)(1,0.616731)(1,0.640992)(1,0.651151)(1,0.662025)(1,0.667398)(1,0.685129)(1,0.708415)(1,0.717854)(1,0.733022)(1,0.757654)(1,0.761221)(1,0.771826)(1,0.787608)(1,0.796893)(1,0.817334)(1,0.835278)(1,0.848640)(1,0.854098)(1,0.870178)(1,0.882773)(1,0.901383)(1,0.911712)(1,0.923024)(1,0.951174)(1,0.962925)(1,0.972402)(1,0.984211)(1,1.009143)(1,1.023554)(1,1.046078)(1,1.056652)(1,1.060173)(1,1.066582)(1,1.092719)(1,1.116148)(1,1.125778)(1,1.159550)(1,1.166293)(1,1.180940)(1,1.207918)(1,1.223011)(1,1.250503)(1,1.255731)(1,1.263745)(1,1.297779)(1,1.315970)(1,1.338086)(1,1.367876)(1,1.387446)(1,1.395226)(1,1.421105)(1,1.433541)(1,1.463520)(1,1.492413)(1,1.503466)(1,1.519716)(1,1.549782)(1,1.576153)(1,1.590588)(1,1.640595)(1,1.673052)(1,1.689443)(1,1.735506)(1,1.748921)(1,1.795284)(1,1.820353)(1,1.831729)(1,1.881249)(1,1.908393)(1,1.947132)(1,1.991260)(1,2.013937)(1,2.074874)
	      (2,2.198030)(2,2.084320)(2,2.021921)(2,1.995515)(2,1.948117)(2,1.917816)(2,1.886626)(2,1.841657)(2,1.824946)(2,0.043048)(2,0.047388)(2,0.049193)(2,1.800200)(2,0.068889)(2,0.052488)(2,0.055629)(2,0.065505)(2,0.062492)(2,0.058463)(2,1.755659)(2,1.741509)(2,1.695630)(2,1.678111)(2,1.651538)(2,0.251781)(2,1.617975)(2,1.580772)(2,1.554991)(2,1.533247)(2,1.508143)(2,1.497382)(2,0.379884)(2,1.463909)(2,1.449049)(2,0.410294)(2,1.426022)(2,1.404309)(2,0.423910)(2,0.440316)(2,1.378070)(2,1.370635)(2,1.348638)(2,1.329711)(2,0.472304)(2,0.474809)(2,1.306790)(2,0.490734)(2,1.274748)(2,1.266559)(2,0.507145)(2,0.519451)(2,0.536774)(2,1.253677)(2,0.544852)(2,1.233493)(2,1.218399)(2,0.555698)(2,0.567175)(2,0.581028)(2,1.190438)(2,1.182246)(2,0.603428)(2,1.167917)(2,0.609453)(2,0.624896)(2,0.641580)(2,1.137783)(2,0.655319)(2,1.124496)(2,1.095074)(2,0.667759)(2,0.679358)(2,0.690717)(2,0.712579)(2,0.722728)(2,1.029649)(2,1.080189)(2,1.052355)(2,1.014665)(2,1.062713)(2,1.072650)(2,0.744165)(2,0.994926)(2,0.765760)(2,0.760244)(2,0.980875)(2,0.781254)(2,0.800998)(2,0.825304)(2,0.960526)(2,0.967925)(2,0.930306)(2,0.919930)(2,0.878573)(2,0.837183)(2,0.791220)(2,0.897654)(2,0.857228)(2,0.906679)(2,0.863896)
	      (3,2.184745)(3,2.067379)(3,2.010921)(3,1.979206)(3,1.935033)(3,1.894008)(3,1.874599)(3,1.823958)(3,1.808896)(3,1.787147)(3,1.727516)(3,1.734941)(3,0.074654)(3,0.042066)(3,0.045997)(3,0.047913)(3,0.053622)(3,0.050284)(3,0.066828)(3,0.060770)(3,1.680035)(3,0.057089)(3,0.062431)(3,1.659581)(3,1.632379)(3,1.603013)(3,1.563172)(3,1.538292)(3,1.514208)(3,1.488375)(3,1.477403)(3,1.434892)(3,1.443965)(3,1.408164)(3,0.371412)(3,1.389168)(3,0.401848)(3,1.364626)(3,0.416901)(3,1.354495)(3,0.431641)(3,1.334278)(3,1.313048)(3,1.293179)(3,0.463109)(3,0.465499)(3,0.480576)(3,1.260154)(3,0.496460)(3,1.248568)(3,1.240426)(3,0.508641)(3,1.205608)(3,1.217919)(3,0.526467)(3,0.535588)(3,1.175112)(3,0.545544)(3,0.555109)(3,1.159289)(3,0.567765)(3,1.149171)(3,0.590436)(3,0.597886)(3,1.118365)(3,1.106453)(3,0.610755)(3,0.630810)(3,0.642895)(3,1.079766)(3,0.655567)(3,1.066832)(3,1.055777)(3,0.679701)(3,0.666993)(3,0.699750)(3,1.036790)(3,1.046686)(3,0.710067)(3,1.012103)(3,0.994521)(3,0.729817)(3,0.983219)(3,0.966333)(3,0.943768)(3,0.766093)(3,0.775138)(3,0.917004)(3,0.810343)(3,0.842565)(3,0.850915)(3,0.752419)(3,0.749721)(3,0.791294)(3,0.956172)(3,0.905470)(3,0.883093)(3,0.891642)(3,0.864968)(3,0.818447)
	      (4,2.184624)(4,2.078995)(4,2.017953)(4,1.995243)(4,1.951012)(4,1.912185)(4,1.885011)(4,1.835383)(4,1.823978)(4,1.798875)(4,0.042246)(4,0.045637)(4,1.752413)(4,1.738931)(4,0.068841)(4,0.048133)(4,0.049470)(4,0.053671)(4,0.066058)(4,0.061705)(4,0.058457)(4,0.056817)(4,1.692773)(4,1.676387)(4,1.643855)(4,1.593714)(4,1.579295)(4,1.552882)(4,1.522725)(4,1.506468)(4,1.495399)(4,0.375692)(4,1.466447)(4,0.407455)(4,1.436360)(4,0.422145)(4,0.438900)(4,1.423941)(4,1.398022)(4,1.390223)(4,1.370600)(4,0.467716)(4,0.472305)(4,1.340755)(4,1.318580)(4,1.300358)(4,0.488564)(4,0.502511)(4,0.517049)(4,1.266232)(4,0.532284)(4,0.542450)(4,1.258221)(4,1.252977)(4,0.555680)(4,0.563610)(4,1.225461)(4,1.210285)(4,0.571938)(4,0.595432)(4,1.183288)(4,0.607355)(4,0.617882)(4,1.168605)(4,1.161831)(4,0.642249)(4,1.128028)(4,1.118280)(4,0.652428)(4,0.686399)(4,0.663181)(4,0.668654)(4,1.094908)(4,0.709832)(4,0.719280)(4,0.734442)(4,1.025591)(4,1.068663)(4,1.048142)(4,1.011087)(4,0.986118)(4,0.798477)(4,0.789144)(4,0.974268)(4,0.818945)(4,0.762645)(4,1.061985)(4,0.924713)(4,0.913529)(4,0.964838)(4,0.952993)(4,0.773366)(4,0.759136)(4,0.836936)(4,0.884535)(4,0.903184)(4,0.871836)(4,1.058765)(4,0.850271)(4,0.855781)
      };
      \draw[color=blue] (axis cs:1.5,0.25) ellipse [x radius=100,y radius=5];
    \end{axis}
  \end{tikzpicture}
  \caption{
	  Eigenvalues of the SCM ($\frac1n\Y\Y^\dagger$), normalized SCM ($\frac1n\Y^{\rm n}{\Y^{\rm n}}^\dagger$), $\Chat$ for $u=u_\stud$ with $t=.1$, and the oracle estimator ($\frac1n\Y^{\rm o}{\Y^{\rm o}}^\dagger$); $N=100$, $c=.2$, $\varepsilon_nn=1$, $\a_1= (\a^1_1,\a^2_1)^\dagger$, $\a^1_1\in\mathbb{R}^{10}$, $\a^2_1\in\mathbb{R}^{90}$, with $\a^1_{1,i}=\sqrt{10}$, $\a^2_{1,i}=0$, such that $\Vert \a_1\Vert^2=N$; $\y_i = \Chalf \x_i$ with $\x_{i,j}$ standard Gaussian and $\C = (16/14.50) \, {\rm diag}( {\bf c}_1, {\bf c}_2)$, ${\bf c}_1 \in \mathbb{R}^{10}$, ${\bf c}_2 \in \mathbb{R}^{90}$, with ${\bf c}_{1i}=1/16$, ${\bf c}_{2i}=1$, such that $\tr\C=N$. Ellipse around the outlier artifact.
  }
  \label{fig:oneoutlier_eigenvalues}
\end{figure}
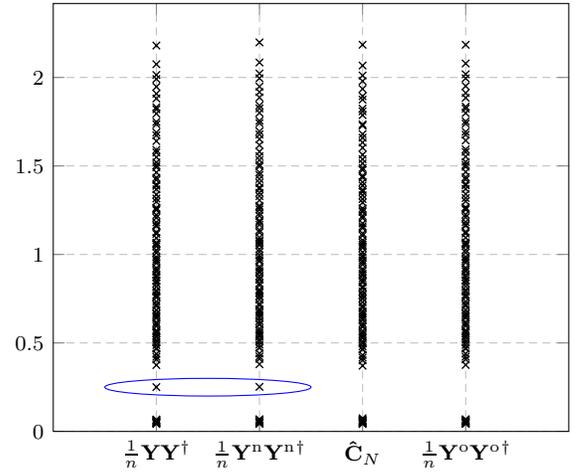

\bigskip

Another interesting case study that shall provide further insight on $\Chat$ is that where the $\a_i$'s (possibly numerous) are independently extracted from a different distribution to that of the $\y_i$'s. This is pursued in the subsequent section.

\section{Random Outliers Scenario}
\label{sec:random}

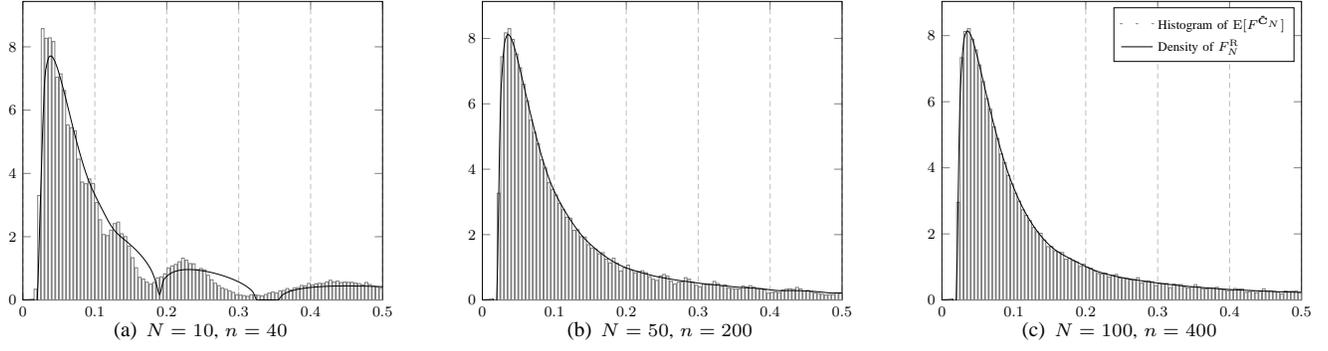
\begin{figure*}[bt]
  \centering
\begin{tabular}{ccc}

   \subfigure [\scriptsize $N=10$, $n=40$]{  
   \adjustbox{width=.3\textwidth}{     
  \begin{tikzpicture}[font=\footnotesize]
    \renewcommand{\axisdefaulttryminticks}{4} 
    \tikzstyle{every major grid}+=[style=densely dashed]       
    \tikzstyle{every axis y label}+=[yshift=-10pt, anchor=near ticklabel] 
    \tikzstyle{every axis x label}+=[yshift=5pt]
    \tikzstyle{every axis legend}+=[cells={anchor=west},fill=white,
        at={(0.98,0.98)}, anchor=north east, font=\scriptsize ]
    /xlabel near ticks
    /ylabel near ticks
    \begin{axis}[
      xmin=0,
      ymin=0,
      xmax=0.5,
      bar width=1.5pt,
      grid=major,
      ymajorgrids=false,
      scaled ticks=true,
      mark repeat=3,
      ]
      \addplot[ybar,mark=none,color=gray,fill=white] coordinates{
(0.002500,0.000000)(0.007500,0.000000)(0.012500,0.000000)(0.017500,0.342000)(0.022500,3.302000)(0.027500,8.576000)(0.032500,8.270000)(0.037500,8.284000)(0.042500,8.168000)(0.047500,7.040000)(0.052500,7.146000)(0.057500,6.620000)(0.062500,5.528000)(0.067500,5.444000)(0.072500,5.344000)(0.077500,4.452000)(0.082500,3.728000)(0.087500,3.674000)(0.092500,3.826000)(0.097500,3.684000)(0.102500,3.078000)(0.107500,2.538000)(0.112500,2.070000)(0.117500,2.036000)(0.122500,2.208000)(0.127500,2.416000)(0.132500,2.462000)(0.137500,2.094000)(0.142500,2.006000)(0.147500,1.702000)(0.152500,1.338000)(0.157500,1.020000)(0.162500,0.716000)(0.167500,0.662000)(0.172500,0.564000)(0.177500,0.500000)(0.182500,0.554000)(0.187500,0.704000)(0.192500,0.734000)(0.197500,0.828000)(0.202500,1.014000)(0.207500,1.078000)(0.212500,1.126000)(0.217500,1.208000)(0.222500,1.322000)(0.227500,1.258000)(0.232500,1.152000)(0.237500,1.144000)(0.242500,0.958000)(0.247500,0.998000)(0.252500,0.956000)(0.257500,0.782000)(0.262500,0.640000)(0.267500,0.530000)(0.272500,0.456000)(0.277500,0.378000)(0.282500,0.336000)(0.287500,0.292000)(0.292500,0.240000)(0.297500,0.168000)(0.302500,0.154000)(0.307500,0.140000)(0.312500,0.118000)(0.317500,0.170000)(0.322500,0.162000)(0.327500,0.150000)(0.332500,0.130000)(0.337500,0.176000)(0.342500,0.220000)(0.347500,0.252000)(0.352500,0.224000)(0.357500,0.288000)(0.362500,0.304000)(0.367500,0.364000)(0.372500,0.322000)(0.377500,0.382000)(0.382500,0.392000)(0.387500,0.462000)(0.392500,0.444000)(0.397500,0.526000)(0.402500,0.472000)(0.407500,0.514000)(0.412500,0.528000)(0.417500,0.520000)(0.422500,0.574000)(0.427500,0.626000)(0.432500,0.574000)(0.437500,0.596000)(0.442500,0.552000)(0.447500,0.568000)(0.452500,0.558000)(0.457500,0.562000)(0.462500,0.534000)(0.467500,0.518000)(0.472500,0.482000)(0.477500,0.558000)(0.482500,0.492000)(0.487500,0.440000)(0.492500,0.418000)(0.497500,0.372000)(0.502500,0.000000)
};      
      \addplot[black,smooth,line width=0.5pt] plot coordinates{
(0.000000,0.006665)(0.005000,0.008734)(0.010000,0.012238)(0.015000,0.019337)(0.020000,0.041263)(0.025000,3.551522)(0.030000,6.804512)(0.035000,7.600604)(0.040000,7.702471)(0.045000,7.509671)(0.050000,7.176730)(0.055000,6.772331)(0.060000,6.330288)(0.065000,5.872332)(0.070000,5.416658)(0.075000,4.978353)(0.080000,4.567598)(0.085000,4.191569)(0.090000,3.858576)(0.095000,3.572921)(0.100000,3.324395)(0.105000,3.095118)(0.110000,2.872545)(0.115000,2.653248)(0.120000,2.444639)(0.125000,2.265590)(0.130000,2.128714)(0.135000,2.021932)(0.140000,1.926884)(0.145000,1.832782)(0.150000,1.734502)(0.155000,1.629249)(0.160000,1.514790)(0.165000,1.388412)(0.170000,1.245803)(0.175000,1.078865)(0.180000,0.868977)(0.185000,0.549135)(0.190000,0.172084)(0.195000,0.599517)(0.200000,0.748511)(0.205000,0.835156)(0.210000,0.890364)(0.215000,0.925949)(0.220000,0.947870)(0.225000,0.959623)(0.230000,0.963446)(0.235000,0.960838)(0.240000,0.952870)(0.245000,0.940316)(0.250000,0.923741)(0.255000,0.903544)(0.260000,0.880011)(0.265000,0.853323)(0.270000,0.823566)(0.275000,0.790728)(0.280000,0.754717)(0.285000,0.715296)(0.290000,0.672110)(0.295000,0.624554)(0.300000,0.571710)(0.305000,0.512059)(0.310000,0.442940)(0.315000,0.358962)(0.320000,0.245174)(0.325000,0.006908)(0.330000,0.002432)(0.335000,0.001838)(0.340000,0.001627)(0.345000,0.001602)(0.350000,0.001814)(0.355000,0.003272)(0.360000,0.142960)(0.365000,0.216071)(0.370000,0.263835)(0.375000,0.299448)(0.380000,0.327469)(0.385000,0.350123)(0.390000,0.368764)(0.395000,0.384235)(0.400000,0.397132)(0.405000,0.407888)(0.410000,0.416822)(0.415000,0.424195)(0.420000,0.430208)(0.425000,0.435023)(0.430000,0.438771)(0.435000,0.441570)(0.440000,0.443510)(0.445000,0.444670)(0.450000,0.445116)(0.455000,0.444903)(0.460000,0.444091)(0.465000,0.442726)(0.470000,0.440833)(0.475000,0.438455)(0.480000,0.435616)(0.485000,0.432351)(0.490000,0.428665)(0.495000,0.424595)(0.500000,0.420147)
};

      
    \end{axis}
  \end{tikzpicture}
  }
}
&

  \subfigure [\scriptsize $N=50$, $n=200$]{  
   \adjustbox{width=.3\textwidth}{     
    \begin{tikzpicture}[font=\footnotesize]
    \renewcommand{\axisdefaulttryminticks}{4} 
    \tikzstyle{every major grid}+=[style=densely dashed]       
    \tikzstyle{every axis y label}+=[yshift=-10pt, anchor=near ticklabel] 
    \tikzstyle{every axis x label}+=[yshift=5pt]
    \tikzstyle{every axis legend}+=[cells={anchor=west},fill=white,
        at={(0.98,0.98)}, anchor=north east, font=\scriptsize ]
    /xlabel near ticks
    /ylabel near ticks
    \begin{axis}[
      xmin=0,
      ymin=0,
      xmax=0.5,
      bar width=1.5pt,
      grid=major,
      ymajorgrids=false,
      scaled ticks=true,
      mark repeat=3,
      ]
      \addplot[ybar,mark=none,color=gray,fill=white] coordinates{
(0.002500,0.000000)(0.007500,0.000000)(0.012500,0.000000)(0.017500,0.006000)(0.022500,3.266000)(0.027500,7.440000)(0.032500,8.170000)(0.037500,8.300000)(0.042500,7.972000)(0.047500,7.524000)(0.052500,7.102000)(0.057500,6.596000)(0.062500,6.090000)(0.067500,5.500000)(0.072500,5.126000)(0.077500,4.776000)(0.082500,4.284000)(0.087500,4.040000)(0.092500,3.588000)(0.097500,3.374000)(0.102500,3.212000)(0.107500,2.948000)(0.112500,2.762000)(0.117500,2.528000)(0.122500,2.504000)(0.127500,2.130000)(0.132500,2.154000)(0.137500,1.938000)(0.142500,1.926000)(0.147500,1.692000)(0.152500,1.588000)(0.157500,1.558000)(0.162500,1.414000)(0.167500,1.450000)(0.172500,1.250000)(0.177500,1.136000)(0.182500,1.278000)(0.187500,1.122000)(0.192500,0.890000)(0.197500,0.992000)(0.202500,1.066000)(0.207500,0.908000)(0.212500,0.834000)(0.217500,0.874000)(0.222500,0.894000)(0.227500,0.836000)(0.232500,0.716000)(0.237500,0.634000)(0.242500,0.600000)(0.247500,0.740000)(0.252500,0.782000)(0.257500,0.730000)(0.262500,0.574000)(0.267500,0.490000)(0.272500,0.546000)(0.277500,0.558000)(0.282500,0.684000)(0.287500,0.636000)(0.292500,0.540000)(0.297500,0.452000)(0.302500,0.406000)(0.307500,0.488000)(0.312500,0.478000)(0.317500,0.480000)(0.322500,0.570000)(0.327500,0.504000)(0.332500,0.420000)(0.337500,0.466000)(0.342500,0.368000)(0.347500,0.352000)(0.352500,0.338000)(0.357500,0.336000)(0.362500,0.428000)(0.367500,0.400000)(0.372500,0.438000)(0.377500,0.394000)(0.382500,0.350000)(0.387500,0.350000)(0.392500,0.310000)(0.397500,0.208000)(0.402500,0.236000)(0.407500,0.250000)(0.412500,0.228000)(0.417500,0.304000)(0.422500,0.330000)(0.427500,0.346000)(0.432500,0.324000)(0.437500,0.386000)(0.442500,0.340000)(0.447500,0.274000)(0.452500,0.296000)(0.457500,0.254000)(0.462500,0.222000)(0.467500,0.202000)(0.472500,0.176000)(0.477500,0.152000)(0.482500,0.164000)(0.487500,0.166000)(0.492500,0.218000)(0.497500,0.200000)(0.502500,0.000000)
};      
      \addplot[black,smooth,line width=0.5pt] plot coordinates{
(0.000000,0.007703)(0.005000,0.010325)(0.010000,0.015019)(0.015000,0.025643)(0.020000,0.073884)(0.025000,5.812084)(0.030000,7.711974)(0.035000,8.112286)(0.040000,8.004870)(0.045000,7.682186)(0.050000,7.262971)(0.055000,6.803031)(0.060000,6.331365)(0.065000,5.864381)(0.070000,5.412301)(0.075000,4.982212)(0.080000,4.579245)(0.085000,4.207993)(0.090000,3.873644)(0.095000,3.580189)(0.100000,3.325724)(0.105000,3.101912)(0.110000,2.899174)(0.115000,2.710762)(0.120000,2.533076)(0.125000,2.364732)(0.130000,2.206298)(0.135000,2.060109)(0.140000,1.928604)(0.145000,1.811642)(0.150000,1.705937)(0.155000,1.607669)(0.160000,1.514833)(0.165000,1.427324)(0.170000,1.345709)(0.175000,1.270253)(0.180000,1.200807)(0.185000,1.137086)(0.190000,1.079153)(0.195000,1.027821)(0.200000,0.983514)(0.205000,0.944747)(0.210000,0.909266)(0.215000,0.875839)(0.220000,0.844133)(0.225000,0.813747)(0.230000,0.784132)(0.235000,0.755377)(0.240000,0.728616)(0.245000,0.705110)(0.250000,0.684629)(0.255000,0.665589)(0.260000,0.646669)(0.265000,0.627666)(0.270000,0.609352)(0.275000,0.592667)(0.280000,0.577675)(0.285000,0.563465)(0.290000,0.549023)(0.295000,0.533884)(0.300000,0.518326)(0.305000,0.503334)(0.310000,0.490137)(0.315000,0.479180)(0.320000,0.469659)(0.325000,0.460321)(0.330000,0.450253)(0.335000,0.439060)(0.340000,0.426812)(0.345000,0.414079)(0.350000,0.401977)(0.355000,0.391821)(0.360000,0.384092)(0.365000,0.377954)(0.370000,0.372133)(0.375000,0.365683)(0.380000,0.358076)(0.385000,0.349044)(0.390000,0.338485)(0.395000,0.326453)(0.400000,0.313317)(0.405000,0.300280)(0.410000,0.290158)(0.415000,0.285780)(0.420000,0.285789)(0.425000,0.286994)(0.430000,0.287652)(0.435000,0.287143)(0.440000,0.285289)(0.445000,0.282055)(0.450000,0.277444)(0.455000,0.271466)(0.460000,0.264114)(0.465000,0.255379)(0.470000,0.245301)(0.475000,0.234123)(0.480000,0.222904)(0.485000,0.214867)(0.490000,0.213560)(0.495000,0.216620)(0.500000,0.220404)
};

    \end{axis}
  \end{tikzpicture}
  }
} &

\subfigure[\scriptsize $N=100$, $n=400$] {
  \adjustbox{width=.3\textwidth}{  
  \begin{tikzpicture}[font=\footnotesize]
    \renewcommand{\axisdefaulttryminticks}{4} 
    \tikzstyle{every major grid}+=[style=densely dashed]       
    \tikzstyle{every axis y label}+=[yshift=-10pt, anchor=near ticklabel] 
    \tikzstyle{every axis x label}+=[yshift=5pt]
    \tikzstyle{every axis legend}+=[cells={anchor=west},fill=white,
        at={(0.98,0.98)}, anchor=north east, font=\scriptsize ]
    /xlabel near ticks
    /ylabel near ticks
    \begin{axis}[
      xmin=0,
      ymin=0,
      xmax=0.5,
      bar width=1.5pt,
      grid=major,
      ymajorgrids=false,
      scaled ticks=true,
      mark repeat=3,
      ]
      \addplot[ybar,mark=none,color=gray,fill=white] coordinates{
(0.002500,0.000000)(0.007500,0.000000)(0.012500,0.000000)(0.017500,0.000000)(0.022500,2.958000)(0.027500,7.334000)(0.032500,8.126000)(0.037500,8.212000)(0.042500,7.892000)(0.047500,7.566000)(0.052500,7.112000)(0.057500,6.612000)(0.062500,6.094000)(0.067500,5.770000)(0.072500,5.242000)(0.077500,4.888000)(0.082500,4.420000)(0.087500,4.154000)(0.092500,3.770000)(0.097500,3.514000)(0.102500,3.230000)(0.107500,2.990000)(0.112500,2.728000)(0.117500,2.560000)(0.122500,2.404000)(0.127500,2.200000)(0.132500,2.066000)(0.137500,2.020000)(0.142500,1.816000)(0.147500,1.614000)(0.152500,1.622000)(0.157500,1.488000)(0.162500,1.452000)(0.167500,1.438000)(0.172500,1.232000)(0.177500,1.266000)(0.182500,1.194000)(0.187500,1.158000)(0.192500,1.020000)(0.197500,1.086000)(0.202500,0.992000)(0.207500,0.964000)(0.212500,0.874000)(0.217500,0.806000)(0.222500,0.786000)(0.227500,0.814000)(0.232500,0.696000)(0.237500,0.750000)(0.242500,0.776000)(0.247500,0.668000)(0.252500,0.666000)(0.257500,0.654000)(0.262500,0.612000)(0.267500,0.616000)(0.272500,0.676000)(0.277500,0.518000)(0.282500,0.554000)(0.287500,0.582000)(0.292500,0.554000)(0.297500,0.438000)(0.302500,0.506000)(0.307500,0.496000)(0.312500,0.502000)(0.317500,0.374000)(0.322500,0.488000)(0.327500,0.466000)(0.332500,0.468000)(0.337500,0.388000)(0.342500,0.366000)(0.347500,0.392000)(0.352500,0.464000)(0.357500,0.400000)(0.362500,0.344000)(0.367500,0.336000)(0.372500,0.358000)(0.377500,0.364000)(0.382500,0.386000)(0.387500,0.328000)(0.392500,0.308000)(0.397500,0.284000)(0.402500,0.294000)(0.407500,0.348000)(0.412500,0.338000)(0.417500,0.338000)(0.422500,0.234000)(0.427500,0.252000)(0.432500,0.272000)(0.437500,0.224000)(0.442500,0.258000)(0.447500,0.342000)(0.452500,0.268000)(0.457500,0.282000)(0.462500,0.260000)(0.467500,0.238000)(0.472500,0.186000)(0.477500,0.274000)(0.482500,0.260000)(0.487500,0.266000)(0.492500,0.272000)(0.497500,0.232000)(0.502500,0.000000)
};      
      \addplot[black,smooth,line width=0.5pt] plot coordinates{
(0.000000,0.007787)(0.005000,0.010448)(0.010000,0.015227)(0.015000,0.026121)(0.020000,0.077266)(0.025000,5.904353)(0.030000,7.746526)(0.035000,8.132547)(0.040000,8.023658)(0.045000,7.705679)(0.050000,7.295002)(0.055000,6.846320)(0.060000,6.388010)(0.065000,5.935590)(0.070000,5.497769)(0.075000,5.079617)(0.080000,4.684783)(0.085000,4.316679)(0.090000,3.978016)(0.095000,3.669369)(0.100000,3.388527)(0.105000,3.132232)(0.110000,2.898209)(0.115000,2.685435)(0.120000,2.492799)(0.125000,2.318000)(0.130000,2.158253)(0.135000,2.011871)(0.140000,1.878626)(0.145000,1.758782)(0.150000,1.652351)(0.155000,1.559222)(0.160000,1.478931)(0.165000,1.409560)(0.170000,1.347614)(0.175000,1.289586)(0.180000,1.233130)(0.185000,1.177146)(0.190000,1.121462)(0.195000,1.066604)(0.200000,1.013560)(0.205000,0.963491)(0.210000,0.917411)(0.215000,0.875925)(0.220000,0.839121)(0.225000,0.806573)(0.230000,0.777436)(0.235000,0.750709)(0.240000,0.725631)(0.245000,0.701888)(0.250000,0.679537)(0.255000,0.658763)(0.260000,0.639656)(0.265000,0.622131)(0.270000,0.605865)(0.275000,0.590319)(0.280000,0.574950)(0.285000,0.559446)(0.290000,0.543811)(0.295000,0.528324)(0.300000,0.513398)(0.305000,0.499334)(0.310000,0.486160)(0.315000,0.473694)(0.320000,0.461739)(0.325000,0.450175)(0.330000,0.438942)(0.335000,0.428049)(0.340000,0.417598)(0.345000,0.407726)(0.350000,0.398463)(0.355000,0.389664)(0.360000,0.381112)(0.365000,0.372661)(0.370000,0.364265)(0.375000,0.355917)(0.380000,0.347599)(0.385000,0.339325)(0.390000,0.331237)(0.395000,0.323660)(0.400000,0.316971)(0.405000,0.311294)(0.410000,0.306361)(0.415000,0.301742)(0.420000,0.297118)(0.425000,0.292378)(0.430000,0.287562)(0.435000,0.282748)(0.440000,0.277949)(0.445000,0.273101)(0.450000,0.268121)(0.455000,0.263003)(0.460000,0.257877)(0.465000,0.253017)(0.470000,0.248737)(0.475000,0.245181)(0.480000,0.242201)(0.485000,0.239476)(0.490000,0.236717)(0.495000,0.233767)(0.500000,0.230608)
};      
\legend{ Histogram of ${\rm E}[F^{\Chat}]$, Density of $F^\rnd_N$ }
    \end{axis}
  \end{tikzpicture}
  }
}

\end{tabular}
\caption{Density of $F^\rnd_N$ versus histogram of ${\rm E}[F^{\Chat}]$ for $\C$ with $[\C]_{ij}=.9^{|i-j|}$, $\D=\I$, $\varepsilon=.05$, and $u(x) = (1+t)/(t+x)$ where $t=.1$.}
\label{fig:validation_density}
\end{figure*}
Assuming $\a_1,\ldots,\a_{\varepsilon_nn}$ to be independent with zero mean and covariance $\D\neq \C$ provides a rather immediate corollary of Theorem~\ref{thm1}, given below. In the results to come, to differentiate between the conditions of Theorem~\ref{thm1} and those of Corollary~\ref{cor1}, we shall use the subscript ``$\rnd$'' standing for ``random outliers scenario''.

\begin{corollary}[Random Outliers] Let Assumption~\ref{as:cN} hold with $\varepsilon >0$ and let $\a_1,\ldots,\a_{\varepsilon_n n}$ be random independent of the $\y_i$'s with $\a_i= \Dhalf \x'_i$, where  $\D \in  \mathbb{C}^{N \times N}$ is deterministic Hermitian positive definite and $\x'_1,\ldots, \x'_{\varepsilon_n n}$ are independent random vectors with i.i.d.\@ zero mean, unit variance, and finite $(8+\eta)$-th order moment entries, for some $\eta>0$. Let us further assume that
$\limsup_N \| \D \C^{-1} \| < \infty$. Then, as $n \to \infty$,
\begin{align*}
\left\| \Chat -  \ShatRnd \right\| \toas 0 
\end{align*}
where
\begin{align*}
 \ShatRnd & \triangleq v\left( \gammaEqRnd \right) \frac{1}{n} \sum_{i=1}^{(1-\varepsilon_n)n}  \y_i \y_i^\dagger + v\left( \alphaEqRnd \right) \frac{1}{n} \sum_{i=1}^{\varepsilon_n n}  \a_i \a_i^\dagger
\end{align*}
with $\gammaEqRnd$ and $\alphaEqRnd$ the unique positive solutions to
\begin{align*}
	\gammaEqRnd &= \frac{1}{N} \tr \C \left(  \frac{(1-\varepsilon) v(\gammaEqRnd)\C}{1+c v(\gammaEqRnd) \gammaEqRnd} + \frac{\varepsilon v(\alphaEqRnd)\D}{1+c v(\alphaEqRnd) \alphaEqRnd} \right)^{-1} \\
\alphaEqRnd &= \frac{1}{N} \tr \D \left( \frac{(1-\varepsilon) v(\gammaEqRnd)\C}{1+c v(\gammaEqRnd) \gammaEqRnd} + \frac{\varepsilon v(\alphaEqRnd)\D}{1+c v(\alphaEqRnd) \alphaEqRnd} \right)^{-1}.
\end{align*}

In particular, for $F_N^{\Chat}(x)$ as defined in Corollary~\ref{cor2},
\begin{align}
F_N^{\Chat} (x) - F_N^\rnd (x) \Rightarrow 0
\notag
\end{align}
almost surely as $n \to \infty$, where $F_N^\rnd (x)$ is a real distribution function with density, defined via its Stieltjes transform
\begin{align}
	m_N^\rnd(z) &= \frac1N\tr\E^{-1} \notag \\
	\E &= \frac{(1-\varepsilon) v(\gammaEqRnd)}{1  +  e_{N,1}(z)}\C  + \frac{\varepsilon v(\alphaEqRnd)}{1  +  e_{N,2}(z)}\D \nonumber
\end{align}
for $z \in \mathbb{C}^+$ and $(e_{N,1}(z),e_{N,2}(z))$ the unique solution in $(\mathbb{C}^+)^2$ to
\begin{align}
	e_{N,1}(z) &= \frac{v(\gammaEqRnd)}n\tr \C  \left( \E  -  z \I  \right)^{ -1} \notag \\
	e_{N,2}(z) &= \frac{v(\alphaEqRnd)}n \tr  \D  \left( \E  -  z \I  \right)^{ -1}.
	\notag
\end{align}
\label{cor1}
\end{corollary}

Figure~\ref{fig:validation_density} shows the density of the distribution ${\rm E}[F_N^{\Chat}]$, obtained from Monte-Carlo averaging, versus $F_N^\rnd$ for different values of $N,n$. It is observed that, as soon as $N$ is of the order of several tens, the asymptotic approximation holds tightly. The (normalized) distance in spectral norm between $\Chat$ and $\Shat^\rnd$ is numerically evaluated in Figure~\ref{fig:validation_spectralNorm} for various values of $N$. As suggested in the second order analysis of \cite{COU14d}, $\|\Chat-\Shat\|$ (or $\|\Chat-\Shat^\rnd\|$ here) is likely to decay at the rate $1/\sqrt{N}$
, which is somewhat confirmed by observing that between $N=20$ and $N=80$, the approximation error decays by a factor of two (precisely, $0.042$ versus $0.019$).

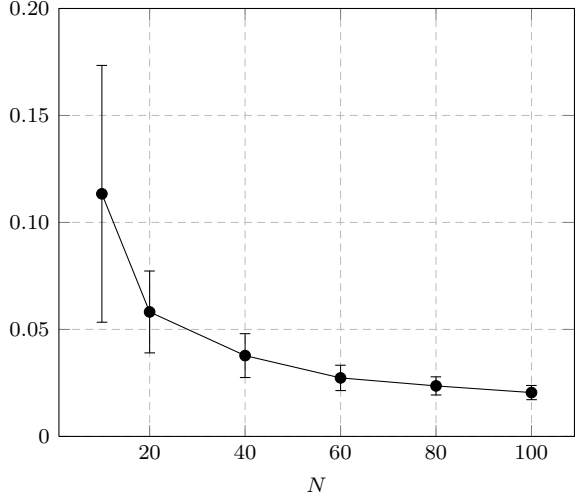
\begin{figure}[t!]
  \centering
  \begin{tikzpicture}[font=\footnotesize]
    \renewcommand{\axisdefaulttryminticks}{4} 
    \tikzstyle{every major grid}+=[style=densely dashed]       
    \tikzstyle{every axis y label}+=[yshift=-10pt, anchor=near ticklabel] 
    \tikzstyle{every axis x label}+=[yshift=5pt]
    \tikzstyle{every axis legend}+=[cells={anchor=west},fill=white,
        at={(0.98,0.98)}, anchor=north east, font=\scriptsize ]
    /xlabel near ticks
    /ylabel near ticks 
   \begin{axis}[
      grid=major,
      xlabel={$N$},
      ytick={0,0.05,0.1,0.15,0.2},
      yticklabels = {$0$,$0.05$,$0.10$,$0.15$,$0.20$},
      ymin=0,
      ymax=0.2
      ]
      \addplot[color=black,mark=*,error bars/.cd, y dir=both, y explicit] coordinates {


(10.000000,0.113377) +- (0,0.06000833275471)
(20.000000,0.058173) +- (0,0.01913112646971)
(40.000000,0.037772) +- (0,0.01024695076596)
(60.000000,0.027348) +- (0,0.0059160797831)
(80.000000,0.023611) +- (0,0.004242640687119)
(100.000000,0.020496) +- (0,0.003316624790355)

    
    };
\end{axis}    
  \end{tikzpicture}
\caption{Mean and standard deviation (error bars) of $\| \Chat -  \ShatRnd \| / \| \Chat \|$ for $\c = 0.25$, $[\C]_{ij}=.9^{|i-j|}$, $[\D]_{ij}=.2^{|i-j|}$, $\eps=.05$, and $u(x) = (1+t)/(t+x)$, with $t=.1$. }
  \label{fig:validation_spectralNorm}
\end{figure}

In the random outliers scenario, $\Chat$ is asymptotically equivalent to the weighted sum of two partial sample covariance matrices, one corresponding to the legitimate data and the other to the outlying data. In the defining equations for $\gamma^\rnd_n$ and $\alpha^\rnd_n$ an interesting symmetrical interplay arises between the weights applied to the legitimate and the outlying data, which are only differentiated by $\varepsilon$. In particular, if $\varepsilon>1/2$, the $\a_i$'s will be considered legitimate (being in majority) and the $\y_i$'s become outliers. 

Despite the symmetrical form of the equations defining $\gamma^\rnd_n$ and $\alpha^\rnd_n$, it remains difficult to extract general insight on these quantities. Thus, again, it is interesting to study the regime where $\varepsilon \to 0$.
In this case, $\gammaEqRnd \to \gamma =\phi^{-1}(1)/(1-c)$, and
\begin{align*}
	\alphaEqRnd & \to \gamma \frac1N \tr \D \C^{-1}.
\end{align*}
As such, the factor dictating the outlier mitigation strength of $\Chat$ is now $\frac1N \tr \D \C^{-1}$. Similar to before, when larger than one, the impact of the outliers will be reduced but these might be enhanced when smaller than one. Interestingly, if $\frac1N\tr\D=\frac1N\tr\C=1$ (say),
both legitimate and outlier samples have similar norm for all large $n$. As such, under this scenario, the SCM $\frac1n\Y\Y^\dagger$ or its normalized version $\frac1n\Y^{\rm n}{\Y^{\rm n}}^\dagger$ behave asymptotically equivalently, neither of which being capable of differentiating between legitimate and outlier data. On the contrary, $\Chat$ is capable of reducing the impact of the outliers as long as $\frac1N \tr \D \C^{-1}>1$. Note here again that $\C$ must be sufficiently distinct from $\I$, which would otherwise entail $\frac1N \tr \D \C^{-1}\simeq 1$ and thus $\Chat$ would be indifferent to outliers. Also, similar to previously, $u$ must be well chosen to avoid enhancing the outlier effect if $\frac1N \tr \D \C^{-1}<1$ (so in particular it is advised that $u$ be similar to $u_\hub$).

Figure~\ref{fig:limit_laws} depicts the previous observations in terms of the deterministic equivalent spectral distributions: $F^\rnd_N$ of $\Chat$, $F^{\rm SCM}_N$ of $\frac1n\Y\Y^\dagger$ (or $F^{\rm nSCM}_N$ of $\frac1n{\Y^{\rm n}}{\Y^{\rm n}}^\dagger$ which satisfies $F^{\rm SCM}_N=F^{\rm nSCM}_N$ here), and $F_N^{\rm oracle}$ of the outlier-free oracle estimator $\frac1n{\Y^{\rm o}}{\Y^{\rm o}}^\dagger$; we take here $\C$ and $\D$ to ensure $\frac1N \tr \D \C^{-1}$ large and $\varepsilon$ is taken small. The sought-for distribution that would optimally discard all outliers is the oracle distribution and, thus, highly robust estimators are expected to have a similar distribution. Figure~\ref{fig:limit_laws} confirms that this is indeed the case of $\Chat$ which shows a close tail behavior but is slightly mismatched in the main distribution lobe. On the contrary, the SCM (normalized or not) shows a strong decay in the main lobe and a non matching tail. The associated theoretical values of $\gamma_n^\rnd$ and $\alpha_n^\rnd$ for $\eps=.05$ are here $v_\hub (\gammaEqRnd) \simeq 1.00$, $v_\hub (\alphaEqRnd) \simeq .1219$, while in the limit $\eps\to 0$, these values become $v_\hub (\gammaEqRnd)\to 1$ and $v_\hub (\alphaEqRnd)\to .1179$.

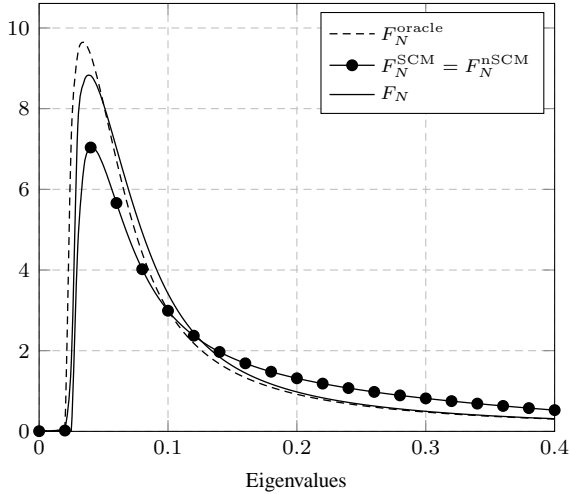
\begin{figure}[t!]
  \centering
  \begin{tikzpicture}[font=\footnotesize]
    \renewcommand{\axisdefaulttryminticks}{4} 
    \tikzstyle{every major grid}+=[style=densely dashed]       
    \tikzstyle{every axis y label}+=[yshift=-10pt, anchor=near ticklabel] 
    \tikzstyle{every axis x label}+=[yshift=5pt]
    \tikzstyle{every axis legend}+=[cells={anchor=west},fill=white,
        at={(0.98,0.98)}, anchor=north east, font=\scriptsize ]
    /xlabel near ticks
    /ylabel near ticks
    \begin{axis}[
      xmin=0,
      ymin=0,
      xmax=0.4,
      bar width=1.5pt,
      grid=major,
      scaled ticks=true,
      mark repeat=4,
      xlabel={Eigenvalues},
      ]
      \addplot[black,smooth,densely dashed,line width=0.5pt] plot coordinates{
(0.000000,0.008923)(0.005000,0.012100)(0.010000,0.017884)(0.015000,0.031358)(0.020000,0.098660)(0.025000,7.333332)(0.030000,9.356286)(0.035000,9.643682)(0.040000,9.334713)(0.045000,8.775691)(0.050000,8.111112)(0.055000,7.412788)(0.060000,6.721767)(0.065000,6.063956)(0.070000,5.455822)(0.075000,4.906373)(0.080000,4.418371)(0.085000,3.989870)(0.090000,3.616002)(0.095000,3.290562)(0.100000,3.007117)(0.105000,2.759633)(0.110000,2.542760)(0.115000,2.351898)(0.120000,2.183159)(0.125000,2.033291)(0.130000,1.899578)(0.135000,1.779756)(0.140000,1.671933)(0.145000,1.574523)(0.150000,1.486191)(0.155000,1.405809)(0.160000,1.332420)(0.165000,1.265207)(0.170000,1.203470)(0.175000,1.146607)(0.180000,1.094098)(0.185000,1.045490)(0.190000,1.000392)(0.195000,0.958457)(0.200000,0.919388)(0.205000,0.882910)(0.210000,0.848800)(0.215000,0.816835)(0.220000,0.786840)(0.225000,0.758655)(0.230000,0.732099)(0.235000,0.707094)(0.240000,0.683481)(0.245000,0.661125)(0.250000,0.640029)(0.255000,0.620026)(0.260000,0.600971)(0.265000,0.582963)(0.270000,0.565894)(0.275000,0.549501)(0.280000,0.533885)(0.285000,0.519195)(0.290000,0.505133)(0.295000,0.491456)(0.300000,0.478438)(0.305000,0.466296)(0.310000,0.454648)(0.315000,0.443114)(0.320000,0.431953)(0.325000,0.421652)(0.330000,0.412061)(0.335000,0.402539)(0.340000,0.392866)(0.345000,0.383554)(0.350000,0.375197)(0.355000,0.367585)(0.360000,0.359941)(0.365000,0.351854)(0.370000,0.343691)(0.375000,0.336296)(0.380000,0.330033)(0.385000,0.324276)(0.390000,0.318161)(0.395000,0.311348)(0.400000,0.304230)(0.405000,0.297750)(0.410000,0.292598)(0.415000,0.288373)(0.420000,0.284079)(0.425000,0.279023)(0.430000,0.273088)(0.435000,0.266758)(0.440000,0.261059)(0.445000,0.256875)(0.450000,0.253954)(0.455000,0.251243)(0.460000,0.247888)(0.465000,0.243517)(0.470000,0.238182)(0.475000,0.232387)(0.480000,0.227178)(0.485000,0.223648)(0.490000,0.221763)(0.495000,0.220435)(0.500000,0.218684)(0.505000,0.216010)(0.510000,0.212249)(0.515000,0.207475)(0.520000,0.202055)(0.525000,0.196863)(0.530000,0.193224)(0.535000,0.191740)(0.540000,0.191480)(0.545000,0.191238)(0.550000,0.190332)(0.555000,0.188478)(0.560000,0.185584)(0.565000,0.181663)(0.570000,0.176850)(0.575000,0.171537)(0.580000,0.166699)(0.585000,0.163846)(0.590000,0.163471)(0.595000,0.164312)(0.600000,0.165118)(0.605000,0.165300)(0.610000,0.164642)(0.615000,0.163075)(0.620000,0.160581)(0.625000,0.157165)(0.630000,0.152865)(0.635000,0.147828)(0.640000,0.142528)(0.645000,0.138264)(0.650000,0.136810)(0.655000,0.137923)(0.660000,0.139810)(0.665000,0.141410)(0.670000,0.142355)(0.675000,0.142546)(0.680000,0.141965)(0.685000,0.140613)(0.690000,0.138495)(0.695000,0.135604)(0.700000,0.131925)(0.705000,0.127463)(0.710000,0.122297)(0.715000,0.116855)(0.720000,0.112692)(0.725000,0.112161)(0.730000,0.114429)(0.735000,0.117196)(0.740000,0.119530)(0.745000,0.121214)(0.750000,0.122222)(0.755000,0.122578)(0.760000,0.122308)(0.765000,0.121433)(0.770000,0.119962)(0.775000,0.117890)(0.780000,0.115201)(0.785000,0.111860)(0.790000,0.107814)(0.795000,0.103016)(0.800000,0.097481)(0.805000,0.091718)(0.810000,0.088407)(0.815000,0.090263)(0.820000,0.093958)(0.825000,0.097364)(0.830000,0.100102)(0.835000,0.102176)(0.840000,0.103642)(0.845000,0.104557)(0.850000,0.104967)(0.855000,0.104905)(0.860000,0.104391)(0.865000,0.103437)(0.870000,0.102046)(0.875000,0.100208)(0.880000,0.097901)(0.885000,0.095094)(0.890000,0.091730)(0.895000,0.087722)(0.900000,0.082948)(0.905000,0.077252)(0.910000,0.070614)(0.915000,0.065786)(0.920000,0.068965)(0.925000,0.073921)(0.930000,0.078055)(0.935000,0.081340)(0.940000,0.083921)(0.945000,0.085919)(0.950000,0.087422)(0.955000,0.088493)(0.960000,0.089175)(0.965000,0.089501)(0.970000,0.089493)(0.975000,0.089168)(0.980000,0.088535)(0.985000,0.087596)(0.990000,0.086348)(0.995000,0.084784)(1.000000,0.082886)};
      \addplot[black,smooth,mark=*,line width=0.5pt] plot coordinates{
(0.000000,0.005069)(0.005000,0.006428)(0.010000,0.008554)(0.015000,0.012300)(0.020000,0.020484)(0.025000,0.052514)(0.030000,4.600278)(0.035000,6.609333)(0.040000,7.037151)(0.045000,6.923871)(0.050000,6.577735)(0.055000,6.133501)(0.060000,5.660644)(0.065000,5.197263)(0.070000,4.764193)(0.075000,4.370867)(0.080000,4.019751)(0.085000,3.709663)(0.090000,3.437848)(0.095000,3.200235)(0.100000,2.990668)(0.105000,2.806113)(0.110000,2.646563)(0.115000,2.501501)(0.120000,2.371852)(0.125000,2.255449)(0.130000,2.150418)(0.135000,2.055196)(0.140000,1.968476)(0.145000,1.889160)(0.150000,1.816324)(0.155000,1.749182)(0.160000,1.687066)(0.165000,1.629401)(0.170000,1.575693)(0.175000,1.525514)(0.180000,1.478492)(0.185000,1.434302)(0.190000,1.392661)(0.195000,1.353320)(0.200000,1.316061)(0.205000,1.280689)(0.210000,1.247035)(0.215000,1.214946)(0.220000,1.184287)(0.225000,1.154937)(0.230000,1.126789)(0.235000,1.099744)(0.240000,1.073717)(0.245000,1.048628)(0.250000,1.024407)(0.255000,1.000991)(0.260000,0.978322)(0.265000,0.954067)(0.270000,0.932918)(0.275000,0.912304)(0.280000,0.892300)(0.285000,0.872864)(0.290000,0.853961)(0.295000,0.835558)(0.300000,0.817625)(0.305000,0.800133)(0.310000,0.783057)(0.315000,0.766374)(0.320000,0.750060)(0.325000,0.734096)(0.330000,0.718463)(0.335000,0.703142)(0.340000,0.688118)(0.345000,0.673375)(0.350000,0.658902)(0.355000,0.644685)(0.360000,0.630714)(0.365000,0.616980)(0.370000,0.603475)(0.375000,0.590192)(0.380000,0.577126)(0.385000,0.564273)(0.390000,0.551631)(0.395000,0.539196)(0.400000,0.526970)(0.405000,0.514950)(0.410000,0.503140)(0.415000,0.491539)(0.420000,0.480150)(0.425000,0.468976)(0.430000,0.458018)(0.435000,0.447281)(0.440000,0.436768)(0.445000,0.426481)(0.450000,0.416455)(0.455000,0.406602)(0.460000,0.397013)(0.465000,0.387663)(0.470000,0.378553)(0.475000,0.369685)(0.480000,0.361058)(0.485000,0.352742)(0.490000,0.344602)(0.495000,0.336705)(0.500000,0.329046)(0.505000,0.321618)(0.510000,0.314414)(0.515000,0.307436)(0.520000,0.300692)(0.525000,0.294188)(0.530000,0.287921)(0.535000,0.281868)(0.540000,0.275994)(0.545000,0.270273)(0.550000,0.264698)(0.555000,0.259298)(0.560000,0.254122)(0.565000,0.249207)(0.570000,0.244543)(0.575000,0.240067)(0.580000,0.235686)(0.585000,0.231319)(0.590000,0.226934)(0.595000,0.222576)(0.600000,0.218354)(0.605000,0.214410)(0.610000,0.210840)(0.615000,0.207623)(0.620000,0.204618)(0.625000,0.201640)(0.630000,0.198533)(0.635000,0.195208)(0.640000,0.191658)(0.645000,0.187974)(0.650000,0.184352)(0.655000,0.181065)(0.660000,0.178350)(0.665000,0.176245)(0.670000,0.174548)(0.675000,0.172962)(0.680000,0.171232)(0.685000,0.169194)(0.690000,0.166761)(0.695000,0.163910)(0.700000,0.160684)(0.705000,0.157213)(0.710000,0.153763)(0.715000,0.150759)(0.720000,0.148638)(0.725000,0.147503)(0.730000,0.147017)(0.735000,0.146722)(0.740000,0.146284)(0.745000,0.145512)(0.750000,0.144310)(0.755000,0.142630)(0.760000,0.140454)(0.765000,0.137785)(0.770000,0.134652)(0.775000,0.131145)(0.780000,0.127487)(0.785000,0.124162)(0.790000,0.121935)(0.795000,0.121273)(0.800000,0.121750)(0.805000,0.122619)(0.810000,0.123393)(0.815000,0.123847)(0.820000,0.123890)(0.825000,0.123485)(0.830000,0.122622)(0.835000,0.121295)(0.840000,0.119499)(0.845000,0.117226)(0.850000,0.114468)(0.855000,0.111223)(0.860000,0.107519)(0.865000,0.103490)(0.870000,0.099597)(0.875000,0.097003)(0.880000,0.096848)(0.885000,0.098405)(0.890000,0.100378)(0.895000,0.102161)(0.900000,0.103564)(0.905000,0.104543)(0.910000,0.105100)(0.915000,0.105246)(0.920000,0.104997)(0.925000,0.104361)(0.930000,0.103344)(0.935000,0.101944)(0.940000,0.100153)(0.945000,0.097954)(0.950000,0.095321)(0.955000,0.092219)(0.960000,0.088600)(0.965000,0.084425)(0.970000,0.079743)(0.975000,0.075161)(0.980000,0.073198)(0.985000,0.075261)(0.990000,0.078414)(0.995000,0.081317)(1.000000,0.083719)};
      \addplot[black,smooth,line width=0.5pt] plot coordinates{
(0.000000,0.007156)(0.005000,0.009325)(0.010000,0.012927)(0.015000,0.019950)(0.020000,0.039316)(0.025000,1.733698)(0.030000,7.566586)(0.035000,8.685195)(0.040000,8.818828)(0.045000,8.554259)(0.050000,8.108567)(0.055000,7.581408)(0.060000,7.024520)(0.065000,6.467021)(0.070000,5.926500)(0.075000,5.413815)(0.080000,4.935702)(0.085000,4.495999)(0.090000,4.096215)(0.095000,3.736125)(0.100000,3.414058)(0.105000,3.127403)(0.110000,2.872959)(0.115000,2.647334)(0.120000,2.447141)(0.125000,2.269220)(0.130000,2.110709)(0.135000,1.969073)(0.140000,1.842108)(0.145000,1.727907)(0.150000,1.624837)(0.155000,1.531498)(0.160000,1.446692)(0.165000,1.369396)(0.170000,1.298728)(0.175000,1.233934)(0.180000,1.174359)(0.185000,1.119441)(0.190000,1.068688)(0.195000,1.021675)(0.200000,0.978027)(0.205000,0.937420)(0.210000,0.899563)(0.215000,0.864204)(0.220000,0.831120)(0.225000,0.800103)(0.230000,0.770996)(0.235000,0.743608)(0.240000,0.717824)(0.245000,0.693526)(0.250000,0.670539)(0.255000,0.648831)(0.260000,0.628310)(0.265000,0.608783)(0.270000,0.590277)(0.275000,0.572790)(0.280000,0.556067)(0.285000,0.540066)(0.290000,0.524978)(0.295000,0.510654)(0.300000,0.496767)(0.305000,0.483426)(0.310000,0.470936)(0.315000,0.459107)(0.320000,0.447484)(0.325000,0.436106)(0.330000,0.425459)(0.335000,0.415637)(0.340000,0.406091)(0.345000,0.396410)(0.350000,0.386889)(0.355000,0.378178)(0.360000,0.370352)(0.365000,0.362740)(0.370000,0.354761)(0.375000,0.346538)(0.380000,0.338822)(0.385000,0.332206)(0.390000,0.326343)(0.395000,0.320361)(0.400000,0.313736)(0.405000,0.306635)(0.410000,0.299868)(0.415000,0.294292)(0.420000,0.289845)(0.425000,0.285623)(0.430000,0.280804)(0.435000,0.275101)(0.440000,0.268817)(0.445000,0.262832)(0.450000,0.258166)(0.455000,0.254938)(0.460000,0.252232)(0.465000,0.249090)(0.470000,0.245011)(0.475000,0.239929)(0.480000,0.234203)(0.485000,0.228720)(0.490000,0.224651)(0.495000,0.222359)(0.500000,0.220954)(0.505000,0.219361)(0.510000,0.216956)(0.515000,0.213497)(0.520000,0.208993)(0.525000,0.203712)(0.530000,0.198359)(0.535000,0.194180)(0.540000,0.192121)(0.545000,0.191625)(0.550000,0.191438)(0.555000,0.190733)(0.560000,0.189141)(0.565000,0.186532)(0.570000,0.182890)(0.575000,0.178310)(0.580000,0.173090)(0.585000,0.168007)(0.590000,0.164494)(0.595000,0.163506)(0.600000,0.164138)(0.605000,0.165009)(0.610000,0.165372)(0.615000,0.164940)(0.620000,0.163615)(0.625000,0.161369)(0.630000,0.158202)(0.635000,0.154140)(0.640000,0.149289)(0.645000,0.144000)(0.650000,0.139296)(0.655000,0.136996)(0.660000,0.137574)(0.665000,0.139373)(0.670000,0.141084)(0.675000,0.142204)(0.680000,0.142589)(0.685000,0.142207)(0.690000,0.141055)(0.695000,0.139139)(0.700000,0.136456)(0.705000,0.132990)(0.710000,0.128740)(0.715000,0.123754)(0.720000,0.118320)(0.725000,0.113599)(0.730000,0.112016)(0.735000,0.113792)(0.740000,0.116545)(0.745000,0.119013)(0.750000,0.120861)(0.755000,0.122036)(0.760000,0.122554)(0.765000,0.122440)(0.770000,0.121719)(0.775000,0.120403)(0.780000,0.118488)(0.785000,0.115963)(0.790000,0.112799)(0.795000,0.108947)(0.800000,0.104354)(0.805000,0.099007)(0.810000,0.093200)(0.815000,0.088851)(0.820000,0.089523)(0.825000,0.093050)(0.830000,0.096584)(0.835000,0.099487)(0.840000,0.101719)(0.845000,0.103329)(0.850000,0.104377)(0.855000,0.104909)(0.860000,0.104963)(0.865000,0.104562)(0.870000,0.103719)(0.875000,0.102440)(0.880000,0.100718)(0.885000,0.098534)(0.890000,0.095860)(0.895000,0.092646)(0.900000,0.088816)(0.905000,0.084257)(0.910000,0.078816)(0.915000,0.072396)(0.920000,0.066418)(0.925000,0.067789)(0.930000,0.072772)(0.935000,0.077126)(0.940000,0.080600)(0.945000,0.083341)(0.950000,0.085473)(0.955000,0.087091)(0.960000,0.088262)(0.965000,0.089037)(0.970000,0.089449)(0.975000,0.089522)(0.980000,0.089275)(0.985000,0.088719)(0.990000,0.087856)(0.995000,0.086686)(1.000000,0.085203)
      };
      \legend{ { $F_N^{\rm oracle}$},{$F_N^{\rm SCM}=F_N^{\rm nSCM}$},{$F_N$} } 
    \end{axis}
  \end{tikzpicture}
  \caption{Density of the approximate (deterministic) spectral distributions for the outlier-free oracle ($F_N^{\rm oracle}$), the SCM or normalized SCM ($F_N^{\rm SCM}=F_N^{\rm nSCM}$), and $\Chat$ ($F_N$), with $u=u_\hub$ with parameter $t=.1$, $[\C]_{ij}=.9^{|i-j|}$, $\D=\I$, $N=100$, $c=.2$, and $\varepsilon=.05$. }
  \label{fig:limit_laws}
\end{figure}

As it appears from Figure~\ref{fig:limit_laws} that the tail of the various estimator distributions may be strongly affected by a weak outlier control, it is interesting to investigate the impact on their moments. For this, we introduce the following application to Corollary~\ref{cor:moments} for the random outlier setting. 
\begin{corollary}[Moments in Random Case]
	Under the setting of Corollary~\ref{cor2}, letting $M_{N,p}^\rnd=\int t^p dF_N^\rnd(t)$, we have
\begin{align}
M_{N,p}^\rnd = \frac{(-1)^p}{p!} \frac1N \tr \T_p^\rnd
\notag
\end{align}
where $\T_p^\rnd$ is obtained recursively as
\begin{align*}
\T_{p+1}^\rnd &=  \sum_{i=0}^p \sum_{j=0}^i \binom{p}{i} \binom{i}{j} \T_{p-i}^\rnd \Q_{i-j+1}^\rnd \T_j^\rnd \\
\Q_{p+1}^\rnd &= (p+1) \left[ (1-\varepsilon) f_{1,p} \R_1 + \varepsilon f_{2,p} \R_2  \right] \\
f_{k, p+1} &= \sum_{i=0}^p \sum_{j=0}^i \binom{p}{i} \binom{i}{j} (p-i+1) f_{k,j} f_{k,i-j} \deltaf_{k,p-i} \\
\deltaf_{k,p+1} &= \frac1n \tr \R_k \T_{p+1}^\rnd ,
\end{align*}
with initial values $\T_0^\rnd = \I$, $f_{k,0} = -1$, $\deltaf_{k,0} = \frac1n \tr \R_k$, and with $\R_1 = v(\gammaEqRnd) \C$, $\R_2 = v(\alphaEqRnd) \D$. 
In particular,
\begin{align*}
&M_{N,1}^\rnd = \frac1N \tr \left[ \varepsilon v(\alphaEqRnd) \D + (1-\varepsilon) v(\gammaEqRnd) \C \right]\\
&M_{N,2}^\rnd = \frac1N \tr \Big[ \left( \varepsilon v(\alphaEqRnd) \D + (1-\varepsilon) v(\gammaEqRnd) \C \right)^2 \nonumber \\
&+\varepsilon v^2(\alphaEqRnd) \D \left[\frac1n \tr \D\right] + (1-\varepsilon) v^2(\gammaEqRnd) \C \left[\frac1n \tr \C\right] \Big].
\end{align*}
	\label{cor:moments_rnd}
\end{corollary}


As expected, $\Chat$ induces a bias in the mean. For fair comparison with the normalized SCM, which estimates $\C$ up to a scale constant, let us define the normalized moments
\begin{align*}
	\bar{M}_{N,p} \triangleq \frac{M_{N,p}}{M_{N,1}}
\end{align*}
and define similarly $\bar{M}_{N,p}^\rnd$ as well as $\bar{M}_{N,p}^{\rm SCM}$ for the SCM, $\bar{M}_{N,p}^{\rm nSCM}$ for the normalized SCM, and $\bar{M}_{N,p}^{\rm oracle}$ for the oracle estimator. Under the same setting as in Figure~\ref{fig:limit_laws}, we provide in the table of Figure~\ref{table:moments} the successive normalized moments and relative error compared to $\bar{M}_{N,p}^{\rm oracle}$. In this case, $\bar{M}_{N,p}^{\rm SCM}=\bar{M}_{N,p}^{\rm nSCM}$. For the scenario at end, given the large support of $F^\rnd_N$, even low order moments tend to take large values so that the asymptotic moment approximation only theoretically holds for $p$ rather small when $N=100$ and we thus only provide these first order moments. The results demonstrate an important advantage brought by $\Chat$ versus the SCM in that the first few order moments are better preserved.

\begin{figure}
	\begin{tabular}{l|ccc}
 & $p=2$ & $p=3$ & $p=4$ \\
\noalign{\smallskip}
\hline
\noalign{\smallskip}
$\bar{M}_{N,p}^{\rm oracle}$ & $9.28$  & $129$ & $1993$ \\  
$\bar{M}_{N,p}^{\rnd}$ (error) & $9.18$ ($1.1\%$) & $126$ ($1.8\%$) & $1945$ ($2.4\%$) \\  
$\bar{M}_{N,p}^{\rm SCM}$ (error) & $8.53$ ($8.2\%$) & $112$ ($13\%$) & $1660$ ($17\%$) \\  
\noalign{\smallskip}
\end{tabular}
\caption{Normalized moments $\bar{M}_{N,p}^\rnd$, $\bar{M}_{N,p}^{\rm SCM}$, versus $\bar{M}_{N,p}^{\rm oracle}$, and relative error $|\, \cdot \, - \bar{M}_{N,p}^{\rm oracle}|/\bar{M}_{N,p}^{\rm oracle}$. Random outliers, $N=100$, $c=.2$, $[\C]_{ij}=.9^{|i-j|}$, $\D = \I$, $\varepsilon=.05$, $u=u_\hub$, $t=.1$. }
\label{table:moments}
\end{figure}

\section{Discussion and Concluding Remarks}
\label{sec:conclusion}

Our study of the robust estimator $\Chat$ in the large random matrix regime has already led to several interesting conclusions, which we shall more thoroughly address in this section. 

Most investigations of robust estimators of scatter focus on the more tractable case where the samples (i.e., the columns of $\Y$) are independent with identical elliptical distribution. The recent results of \cite{Couillet2013,COU14} have revealed that, as $u(x)$ gets close to the Tyler $1/x$ function, in the large random matrix regime, $\Chat$ tends to behave similar to the normalized SCM defined in \eqref{eq:normalized_SCM}. This conclusion was quite pessimistic as it suggested no real improvement of $\Chat$ over simplistic alternative robust methods. In the concluding remarks of \cite[Section~4]{COU14}, the authors anticipated a change of behavior of $\Chat$ versus the normalized SCM for deterministic outlier data. This was revealed here both in Section~\ref{sec:few_outliers} and in Section~\ref{sec:random} where it is made clear that, unlike the normalized SCM, the robust estimators of scatter smartly detect the outliers, essentially by evaluating and comparing the quadratic forms $\y^\dagger \C^{-1}\y$ for each column vector $\y$ of $\Y$. Larger $\y^\dagger \C^{-1}\y$ imply more attenuation of $\y$ within the observed samples. However, an incidental consequence of this behavior of $\Chat$ is that small values of $\y^\dagger \C^{-1}\y$ enhance the effect of $\y$ even though it might not comply with the legitimate sample distribution, thus increasing the probability of inducing false alarms. This has led us to conclude that the function $u$ should be adequately tuned to avoid such a phenomenon. Another consequence is that matrices $\Chat$ with legitimate data of covariance $\C$ close to the identity will have very poor outlier rejection properties.

When the outliers are few, the empirical spectral measure $F^{\Chat}$ of $\Chat$ is asymptotically the same as that of the SCM, normalized SCM, and oracle estimators. As such, if one's interest is on functionals of the eigenvalues of $\C$, such as moments, and only few outliers are expected, sophisticated robust estimators come to no avail. 
This being said, the outliers may naturally engender extra isolated eigenvalues (only finitely many) in the spectrum of $\frac1n\Y\Y^\dagger$ which $\Chat$ might suitably remove while the normalized SCM may not (recall Figure~\ref{fig:oneoutlier_eigenvalues}). For subspace detection and estimation applications, where the information often lies in the eigenvectors of isolated eigenvalues, discarding such outlying information is critical and thus robust estimators may bring important performance gains. For instance, applications in finance and biostatistics (where data are often assumed to contain outliers) heavily rely on isolated eigenvalue-eigenvector pairs, see e.g., \cite{POT00,QUA13}. The experimenter must however keep in mind that, according to our analysis, $\Chat$ is most effective at automatically suppressing \emph{isolated} outliers (the less of these relative to the legitimate samples the better) and loses discriminatory power as the outliers approach one another.

The observation made in Section~\ref{sec:random} that the distribution (in particular through its first order moments) $F_N^\rnd$ is much closer to the oracle estimator than would the (normalized or not) SCM be leads to some interesting applications when it comes to designing improved estimators for $\C$ that both account for the fact that $n$ is not large compared to $N$ and for the fact that the observed data are prone to outliers. Such investigations were successively made in \cite{CHE11} for the finite $N,n$ regime and later in \cite{COU14} for the large $N,n$ regime where hybrid Ledoit--Wolf \cite{LED04} and Tyler \cite{Tyler1987} estimators were proposed that improve the estimation of $\C$ by providing an extra degree of freedom (a regularization parameter) which is selected so to minimize the expected Frobenius norm between $\C$ and the estimator under study. Since the Frobenius norm is nothing but a functional of second order moments, the observation made in the table of Figure~\ref{table:moments} strongly suggests that the Ledoit--Wolf estimator alone (being based on the SCM) would be quite sensitive to deterministic outliers while the estimators studied in \cite{CHE11,COU14}, which are essentially of a similar class as $\Chat$, would be much more resilient to such outliers.

When the number of outliers is much larger, even in the random outlier scenario studied in Section~\ref{sec:random}, very little can be said. However, we noticed an interesting symmetry in the equations defining the weights $\gamma_n^\rnd$ and $\alpha_n^\rnd$ of Corollary~\ref{cor1}, which reveals that the asymptotic proportion $\varepsilon$ of outliers versus $1-\varepsilon$ of legitimate data could tip for $\varepsilon>.5$ towards letting the outliers be considered as the truly legitimate data.

\bigskip

In summary, the present study provides a first step towards a better understanding of the behavior of (classical) robust estimators of scatter against arbitrary outliers. Our findings underline several key aspects of such estimators of profound practical relevance, such as the importance of the population covariance matrix $\C$ of the legitimate data in the rejection power of the estimator, as well as the risks inherent to using weight functions $u$ of the Tyler type. Nonetheless, this study remains at the theoretical level of the estimator itself and does not consider the implications when used as a plug-in estimator in detection or estimation methods. Whether these methods are based on local information (isolated eigenvalue, specific eigenvectors, etc.) or global information (functional of the eigenvalues, projections on large subspaces, etc.) about $\C$ will entail significant differences in the way $\Chat$, through the weight function $u$, must be tailored. Such considerations are left to future investigations.


\appendices




\section{Proof of Theorem \ref{thm1}}
\label{app:det}

The main technical difficulty of the article lies in the proof of Theorem~\ref{thm1} which extends the methods developed in \cite{Couillet2013} to multiple sample types. The present section is dedicated to this proof. Some auxiliary random matrix results will be then listed in Appendix~\ref{app:RandomMatrixResults}, while Appendix~\ref{app:moments} will deal with the (rather immediate) proof of Corollary~\ref{cor:moments}.

\bigskip

The proof of Theorem~\ref{thm1} is divided in two parts. First, we show that the system of fixed-point equations \eqref{gammaThm} admits a unique vector solution and that such solution is bounded as $n\to\infty$. This then defines unequivocally the matrix $\Shat$. We then show in a second part that $\| \Chat - \Shat \| \toas 0$.

\subsection{Existence, uniqueness, boundedness of the solution to \eqref{gammaThm}}
\label{app:boundedness}
To prove existence and uniqueness, we use the framework of standard interference functions \cite{Yates1995}.

\begin{definition}
A function $\h = (h_0,\ldots,h_s): \mathbb{R}_+^{1+s} \to \mathbb{R}_+^{1+s}$ is a standard interference function if it satisfies the conditions:
\begin{enumerate}
	\item {\bf Positivity}: if $q_0,\ldots,q_s \geq 0$, then $h_i(q_0,\ldots,q_s)>0$ for all $i$.
	\item {\bf Monotonicity}: if $q_0 \geq q_0^\prime,\ldots,q_s \geq q_s^\prime$ then, for all $i$, $h_i(q_0,\ldots,q_s) \geq h_i(q_0^\prime,\ldots,q_s^\prime)$.
	\item {\bf Scalability}: for all $\delta > 1$ and all $i$, $\delta h_i(q_0,\ldots,q_s) > h_i(\delta q_0,\ldots,\delta q_s)$.
\end{enumerate}
\label{def:intFunction}
\end{definition}
By \cite[Thm.~2]{Yates1995}, if $\h$ is a standard interference function for which there exists $(q_0,\ldots,q_s)$ such that $q_i \geq h_i(q_0,\ldots,q_s)$ for all $i$, then the system of equations $q_i = h_i(q_0,\ldots,q_s)$, $i=0,\ldots,s$, has a unique solution. 

Define $\h \triangleq (h_0,\ldots,h_{\eps n}): \mathbb{R}_+^{1+\eps n} \to \mathbb{R}_+^{1+\eps n}$ with
\begin{align*}
& h_0(q_0,\ldots,q_{\eps n}) = \nonumber \\
& \frac{1}{N} \tr \C  \left(  \frac{(1-\varepsilon) v(q_0)}{1+c v(q_0) q_0} \C + \frac{1}{n} \sum_{j=1}^{\varepsilon_n n} v\left( q_j \right)  \a_j \a_j^\dagger  \right)^{ -1} \\
& h_i(q_0,\ldots,q_{\eps n}) = \nonumber \\
& \frac{1}{N}  \a_i^\dagger \left( \frac{(1-\varepsilon) v(q_0)}{1+c v(q_0) q_0} \C + \frac{1}{n} \sum_{j\neq i}
 v\left( q_j \right)  \a_j  \a_j^\dagger \right)^{ -1}  \a_i
\end{align*}
for $i=1,\ldots,\eps n$. 
Let us prove that $\h$ meets the conditions of Definition~\ref{def:intFunction} and that, for $i=0,\ldots,\eps n$, $h_i(q_0,\ldots,q_{\eps n}) \leq q_i$ for some $(q_0,\ldots,q_{\eps n})$, which will then prove existence and uniqueness.

From Assumption \ref{as:C} and the fact that $v$ is bounded, we clearly have $h_i > 0$ for all $i$. To show monotonicity, let us first define 
\begin{align*}
\B(q_0,\ldots,q_{\eps n}) = \frac{(1-\varepsilon) v(q_0)}{1+c v(q_0) q_0} \C + \frac{1}{n} \sum_{j=1}^{\varepsilon_n n} v\left( q_j \right)  \a_j \a_j^\dagger
\end{align*}
and take $q_0,\ldots,q_{\eps n}$ and $q_0^\prime,\ldots,q_{\eps n}^\prime$ such that $q_i \geq q_i^\prime$ for all $i$. Then, since $v$ is non-increasing and $\psi(x)=x v(x)$ is increasing,
\begin{align*}
\B(q_0,\ldots,q_{\eps n}) \preceq \B(q_0^\prime,\ldots,q_{\eps n}^\prime) .
\end{align*}
From \cite[Cor.~7.7.4]{Horn1985}, this implies
\begin{align*}
\left( \B(q_0,\ldots,q_{\eps n}) \right)^{-1} \succeq \left( \B(q_0^\prime,\ldots,q_{\eps n}^\prime) \right)^{-1}
\end{align*}
from which $h_0 (q_0,\ldots,q_{\eps n}) \geq h_0 (q_0^\prime,\ldots,q_{\eps n}^\prime) $. By the same arguments, $h_i (q_0,\ldots,q_{\eps n}) \geq h_i (q_0^\prime,\ldots,q_{\eps n}^\prime)$ for $i=1,\ldots,\eps n$, thus proving the monotonicity of $\h$. Finally, to show scalability, let us rewrite $h_0$ as
\begin{align*}
& h_0(q_0,\ldots,q_{\eps n}) = \nonumber \\
&  \frac{1}{N} \tr \C  \left( (1-\varepsilon)  \frac{\Theta(q_0)}{q_0} \C + \frac{1}{n} \sum_{j=1}^{\varepsilon_n n} \frac{\psi\left( q_j \right)}{q_j}  \a_j \a_j^\dagger  \right)^{ -1} 
\end{align*}
where $\Theta(x)=\frac{\psi(x)}{1+c \psi(x)}$. Since $\psi(x)$ is increasing, so is $\Theta(x)$ and, for any $\delta > 1$,
\begin{align*}
& h_0(\delta q_0,\ldots,\delta q_{\eps n}) \nonumber \\
&   = \frac{\delta}{N} \tr \C  \left(  \frac{(1-\varepsilon) \Theta(\delta q_0)}{q_0} \C + \frac{1}{n} \sum_{j=1}^{\varepsilon_n n} \frac{\psi\left( \delta q_j \right)}{q_j}  \a_j \a_j^\dagger  \right)^{ -1} \\
 &  < \delta h_0(q_0,\ldots,q_{\eps n}) .
\end{align*}
We show similarly $h_i(\delta q_0,\ldots,\delta q_{\eps n}) < \delta h_i(q_0,\ldots,q_{\eps n})$ for $i=1,\ldots,\eps n$, thus proving the scalability of $\h$.

Thus, $\h$ is a standard interference function and it remains to show that $h_i(q_0,\ldots,q_{\eps n}) \leq q_i$ for some $(q_0,\ldots,q_{\eps n})$ and for all $i$. For $i=0$,
\begin{align*}
h_0(q_0,\ldots,q_{\eps n}) = \frac1N \tr  \C \left( \B(q_0,\ldots,q_{\eps n}) \right)^{-1}
\end{align*}
where
\begin{align*}
 \B(q_0,\ldots,q_{\eps n}) \succeq  \frac{(1-\varepsilon) v(q_0)}{1+c v(q_0) q_0} \C
\end{align*}  
and thus, by definition of $\psi$,
\begin{align}
h_0(q_0,\ldots,q_{\eps n}) \leq \frac{1+c \psi(q_0)}{(1-\varepsilon)\psi(q_0)} q_0 .
\label{bound_gamma}
\end{align}
As a consequence, we need to find some $q_0$ for which $\frac{1+c \psi(q_0)}{(1-\varepsilon)\psi(q_0)} \leq 1$ or, equivalently, $\psi(q_0) \geq \frac{1}{1-\varepsilon -c}$. Such a choice of $q_0$ is always possible since $\psi$ is increasing on $[0,\infty)$ with image $[0,\psi_{\infty})$ where $\frac{1}{1-\varepsilon-c} < \psi_{\infty}$ (this unfolds from $\phi_{\infty} > \frac{1}{1-\varepsilon}$). Therefore, for any $q_0$ such that $\frac{1}{1-\varepsilon-c} \leq \psi(q_0) < \psi_{\infty}$, we have $h_0(q_0,\ldots,q_{\eps n}) \leq q_0$. Take for instance $q_0 = \psi^{-1}(\frac{1}{1-\varepsilon-c})$ and consider now the functions $h_i$, $i=1,\ldots,\eps n$ for which, using \cite[Lemma 10]{Wagner2012} and similar arguments as above,
\begin{align}
h_i(q_0,\ldots,q_{\eps n}) & \leq  q_0   \frac{1+c \psi(q_0)}{(1-\varepsilon)\psi(q_0)} \frac1N \a_i^\dagger \C^{-1} \a_i    \notag \\
&= q_0 \frac1N \a_i^\dagger \C^{-1} \a_i \triangleq w_i .
\label{bound_alphas}
\end{align}
Therefore, taking $q_i = w_i$ for $i=1,\ldots,\eps n$, we also have $h_i(q_0,\ldots,q_{\eps n}) \leq q_i$. Altogether, we have shown that the function $\h$ satisfies the conditions of \cite[Thm.~2]{Yates1995} implying that there exists a unique solution to \eqref{gammaThm}. As such, $\Shat$ as introduced in the statement of Theorem~\ref{thm1} is well-defined.

\bigskip

We now turn our focus to the boundedness of the solution to \eqref{gammaThm}. From (\ref{bound_gamma}) and (\ref{bound_alphas}), along with Assumption \ref{as:C}, we immediately have that $(\gammaEq,\alpha_{1,n},\ldots,\alpha_{\varepsilon_n n,n})$ is uniformly bounded in $n$, i.e., $\limsup_n \gammaEq < \infty$ and $\limsup_n \max_{1\leq i \leq \eps n} \alphaEqi < \infty$. Furthermore, $\gammaEq$ can be shown to be uniformly away from zero as follows. By monotonicity of the $\h$ function, $h_0(q_0,\ldots,q_{\eps n}) \geq h_0(0,\ldots,0)$, i.e.,
\begin{align*}
h_0(q_0,\ldots,q_{\eps n}) & \geq \frac{1}{v(0)} \frac1N \tr \H_N^{-1} 
 \geq \frac{1}{v(0)} \frac{1}{\| \H_N \|} , \notag
\end{align*}
where the matrix $\H_N$ is defined as
\begin{align*}
\H_N \triangleq (1-\varepsilon) \I + \frac{1}{n} \sum_{j=1}^{\varepsilon_n n}  \Chalfm \a_j \a_j^\dagger \Chalfm.
\end{align*}
By Assumption \ref{as:C} we have $\limsup_n \| \H_N \| < \infty$ and, consequently, $\liminf_n \gammaEq > 0$.



\subsection{Convergence of $\Chat-\Shat$}

Having proved that $\Shat$ is well defined, we now turn to the core of the proof of Theorem~\ref{thm1}. The outline of the proof follows tightly that of \cite[Thm.~2]{Couillet2013} but for a model that is (i) simpler in its assuming the legitimate data to be essentially Gaussian instead of elliptical, but (ii) made more complex due to the deterministic addition of the vectors $\a_1,\ldots,\a_{\eps n}$. Our way to deal with (ii) is by controlling in parallel the quantities asymptotically approximated by $\gammaEq$ and those asymptotically approximated by $\alphaEqi$. Since some parts of the proof mirror closely those in \cite[Thm.~2]{Couillet2013}, we shall mainly focus on the significantly differing aspects.

First note that we can assume $\C = \I$ by studying $\Chalfm \Chat \Chalfm$ instead of $\Chat$, in which case we have $\Chalfm \a_i$ in place of the original $\a_i$. This can be seen from (\ref{Maronna}), the implicit equation solved by $\Chat$. Hence, from now on we assume $\C = \I$ without loss of generality. Using the definition $v(x) \triangleq u \left( g_n^{-1}(x) \right)$, with $g_n(x) = x/(1-\c \phi (x))$, and following the same steps as
in \cite{Couillet2013}, let us write
\begin{align*}
\Chat &= \frac{1}{n} \sum_{i=1}^{(1-\eps)n} v \left( d_i  \right) \x_i \x_i^\dagger + \frac{1}{n} \sum_{i=1}^{\eps n} v \left( b_i  \right) \a_i \a_i^\dagger
\end{align*}
with $d_i \triangleq \frac{1}{N} \x_i^\dagger \Chatrx^{-1} \x_i$ and $b_i \triangleq \frac{1}{N} \a_i^\dagger \Chatra^{-1} \a_i$, where $\Chatrx \triangleq \Chat - v \left( d_i  \right) \x_i \x_i^\dagger$ and $\Chatra \triangleq \Chat - v \left( b_i \right) \a_i \a_i^\dagger$.
Further define
\begin{align*}
e_i &\triangleq \frac{v(d_i)}{v(\gammaEq)} , \qquad f_i  \triangleq \frac{v(b_i)}{v(\alphaEqi)} ,
\end{align*}
with $\gammaEq$ and $\alphaEqi$ as in the statement of Theorem~\ref{thm1} but for $\C=\I$, i.e., $\gammaEq$ and $\alphaEqi$ are the positive solutions to
\begin{align*}
 \gammaEq & = \frac{1}{N} \tr \left(  \frac{(1-\varepsilon) v(\gammaEq)}{1+c v(\gammaEq) \gammaEq} \I + \frac{1}{n} \sum_{j=1}^{\eps n} v\left( \alphaEqj \right)  \a_j  \a_j^\dagger  \right)^{\hspace{-1mm}  -1} \\ 
 \alphaEqi & = \frac{1}{N}  \a_i^\dagger \left( \frac{(1-\varepsilon) v(\gammaEq)}{1+c v(\gammaEq) \gammaEq} \I + \frac{1}{n}  \sum_{j\neq i}  v\left( \alphaEqj \right)  \a_j  \a_j^\dagger \right)^{\hspace{-1mm} -1}  \hspace{-2mm} \a_i .
\end{align*}

The core of the proof is to show that
\begin{align} \label{max_ei}
 \max_{1\leq i \leq (1-\eps)n} | e_i -1 | & \toas 0 \\ \label{max_fi}
 \max_{1\leq i \leq \eps n} | f_i -1 | & \toas 0  .
\end{align}
Let us first relabel $e_i$ and $f_i$ such that $e_1 \leq \ldots \leq e_{(1-\eps)n}$ and $f_1 \leq \ldots \leq f_{\eps n}$ and denote $\deltap_n = \max (e_{(1-\eps) n}, f_{\eps n})$. For any $i = 1,\ldots, (1-\eps)n$, we have
\begin{align*}
e_i  & = \frac{v \left(   \frac{1}{N} \x_i^\dagger \left( \frac{1}{n} \hspace{-0.5mm} \sum\limits_{j\neq i}  v( d_j ) \x_j \x_j^\dagger
 + \frac{1}{n} \hspace{-0.5mm} \sum\limits_{j=1}^{\eps n} v( b_j ) \a_j \a_j^\dagger  \right)^{\hspace{-1mm} -1} \hspace{-2mm} \x_i  \right)}{v(\gammaEq)} \\
& \hspace{-1mm} \leq \frac{v  \left(  \hspace{-0.5mm}  \frac{1}{\deltap_n N} \x_i^\dagger   \left(  \frac{1}{n} \hspace{-0.5mm}  \sum\limits_{j\neq i}  v( \gammaEq ) \x_j \x_j^\dagger
 + \frac{1}{n} \hspace{-0.5mm} \sum\limits_{j=1}^{\eps n} v ( \alphaEqj ) \a_j \a_j^\dagger  \right)^{ \hspace{-1mm} -1}  \hspace{-2mm} \x_i  \hspace{-0.5mm}   \right)}{v(\gammaEq)}
\end{align*}
where we used $v(d_j)=v(\gamma_n)e_j$, $v(b_j)=v(\alpha_{j,n})f_j$ and the inequality arises from $e_j,f_j\leq \delta_n$, from $v$ being non-increasing, and from \cite[Cor.~7.7.4]{Horn1985}. For readability, let
\begin{align*}
\Ani \triangleq \frac{1}{n}   \sum\limits_{j \neq i}  v( \gammaEq ) \x_j \x_j^\dagger + \frac{1}{n} \sum\limits_{j=1}^{\eps n} v ( \alphaEqj ) \a_j \a_j^\dagger .
\end{align*}
From the random matrix result, Lemma~\ref{lemma1} of Appendix~\ref{app:RandomMatrixResults},
\begin{align*}
\max_{1 \leq i \leq (1-\eps) n}  \left|  \frac{1}{N} \x_i^\dagger \AniI \x_i    - \gammaEq \right| \toas 0 .
\end{align*}
Thus, for $\zeta >0$, with probability one, we have for all large $n$
\begin{align} \label{ei_bound}
e_{(1-\eps)n} \leq \frac{v\left( \frac{1}{\deltap_n} (\gammaEq - \zeta)\right)}{v(\gammaEq)} .
\end{align} 
We can proceed similarly to bound $f_i$ from above as
\begin{align*}
f_i  & \leq \frac{v\left(  \frac{1}{\deltap_n N} \a_i^\dagger \BniI \a_i   \right)}{v(\alphaEqi)} 
\end{align*}
for any $i=1,\ldots,\eps n$, with
\begin{align*}
\Bni \triangleq \frac{1}{n} \sum_{j=1}^{(1-\eps)n} v\left( \gammaEq  \right) \x_j \x_j^\dagger + \frac{1}{n} \sum_{j\neq i} v\left( \alphaEqj  \right) \a_j \a_j^\dagger
\end{align*}
and we now use Lemma~\ref{lemma2} in Appendix~\ref{app:RandomMatrixResults} which states
\begin{align*}
\max_{1 \leq i \leq \eps n}  \left|  \frac{1}{N} \a_i^\dagger \BniI \a_i      - \alphaEqi \right| \toas 0 .
\end{align*}

Therefore, for the same $\zeta >0$ and for all large $n$ a.s.,
\begin{align} \label{fi_bound}
f_{\eps n} \leq \frac{v\left( \frac{1}{\deltap_n} (\alphaEqi - \zeta)\right)}{v(\alphaEqi)} .
\end{align}

We now consider separately the subsequence of $n$ over which $e_{(1-\eps)n} \geq f_{\eps n}$ and that over which $e_{(1-\eps)n} < f_{\eps n}$ (these subsequences may be empty or finite).

\subsubsection*{Subsequence $e_{(1-\eps)n} \geq f_{\eps n}$}
On this subsequence, \eqref{ei_bound} becomes
\begin{align*}
e_{(1-\eps) n} \leq \frac{v\left( \frac{1}{e_{(1-\eps) n}} (\gammaEq - \zeta)\right)}{v(\gammaEq)}
\end{align*} 
or alternatively, since $e_{(1-\eps) n}$ is positive, 
\begin{align*}
1 \leq \frac{\psi \left( \frac{\gamma_n}{e_{(1-\epsilon)n}} \left( 1 - \frac{\zeta}{\gamma_n} \right)  \right)}{\psi(\gamma_n) \left( 1 - \frac{\zeta}{\gamma_n} \right)}.
\end{align*}
We want to prove that, for any $\ell > 0$, $e_{(1-\epsilon)n} \leq 1 + \ell$ for all large $n$ a.s. Let us assume the opposite, i.e., $e_{(1-\epsilon)n} > 1 + \ell$ infinitely often, and let us restrict ourselves to a (further) subsequence where this always holds. Then,
\begin{align}
1 \leq \frac{\psi \left( \frac{\gammaEq}{1 + \ell} \left( 1 - \frac{\zeta}{\gammaEq} \right)  \right)}{\psi(\gammaEq) \left( 1 - \frac{\zeta}{\gammaEq} \right)} \leq \frac{\psi \left( \frac{\gammaEq}{1 + \ell}  \right)}{\psi(\gammaEq) \left( 1 - \frac{\zeta}{\gammaEq} \right)} .
\notag
\end{align}
From the uniform boundedness of $\gammaEq$ away from zero and infinity (see Appendix~\ref{app:boundedness}), considering yet a further subsequence over which $\gammaEq \to \gamma_0 > 0$, we obtain in the limit
\begin{align*}
\psi (\gamma_0) \left( 1 - \frac{\zeta}{\gamma_0} \right) \leq  \psi \left( \frac{\gamma_0}{1 + \ell}  \right).
\end{align*}
This being valid for each $\zeta > 0$, a contradiction is raised in the limit $\zeta \to 0$. Therefore, either the subsequence over which $e_{(1-\eps)n} \geq f_{\eps n}$ is finite or $e_{(1-\epsilon)n} \leq 1 + \ell$ for all large $n$ a.s. Assuming the former, then $e_{(1-\eps)n} < f_{\eps n}$ for all large $n$, which is considered next.


\subsubsection*{Subsequence $e_{(1-\eps)n} < f_{\eps n}$}
On this subsequence, \eqref{fi_bound} becomes
\begin{align} 
f_{\eps n} \leq \frac{v\left( \frac{1}{f_{\eps n}} (\alpha_{\eps n, n} - \zeta)\right)}{v(\alpha_{\eps n, n})} 
\label{palawan}
\end{align}
for all large $n$ a.s. Again, we wish to prove that with, say, the same $\ell > 0$ as above, $f_{\eps n} \leq 1 + \ell$ for all large $n$ a.s. Consider first the case $\liminf_n \alpha_{\eps n,n} = 0$ and restrict ourselves to those converging subsequences over which $\alpha_{\eps n, n} \to 0$. In the limit, $v(\alpha_{\eps n,n}) \to v(0)$ so that, for any $\theta>0$ and for $n$ large enough, $v(\alpha_{\eps n,n}) > v(0)-\theta$. This, along with $v(1/f_{\eps n} (\alpha_{\eps n, n} - \zeta) ) \leq v(0)$ gives $f_n \leq v(0)/(v(0)-\theta)$ for all large $n$ implying that, for any $\ell >0$, $f_n \leq 1+ \ell$ for all large $n$ a.s.
Consider now the rest of subsequences for which $\liminf_n \alpha_{\eps n,n} > 0$ and rewrite (\ref{palawan}) as
\begin{align*}
1 \leq \frac{\psi \left( \frac{\alpha_{\eps n, n}}{f_{\eps n}} \left( 1 - \frac{\zeta}{ \alpha_{\eps n, n} } \right)  \right)}{\psi(\alpha_{\eps n, n}) \left( 1 - \frac{\zeta}{\alpha_{\eps n, n}} \right)} .
\end{align*}
As above for $e_{(1-\eps)n}$, we assume $f_{\eps n} > 1 + \ell$ infinitely often, and restrict ourselves to a further subsequence where this holds for all $n$. Then,
\begin{align*} 
1 \leq \frac{\psi \left( \frac{\alpha_{\eps n, n}}{1 + \ell} \right)}{\psi(\alpha_{\eps n, n}) \left( 1 - \frac{\zeta}{\alpha_{\eps n, n}} \right)} .
\end{align*}
From the boundedness of $\alpha_{\eps n,n}$ (see Appendix~\ref{app:boundedness}), we can take a converging (further) subsequence over which $\alpha_{\eps n, n} \to \alpha_0 > 0$. In the limit,
\begin{align*} 
\psi (\alpha_0) \left( 1 - \frac{\zeta}{\alpha_0} \right) \leq  \psi \left( \frac{\alpha_0}{1 + \ell}  \right)
\end{align*}
which is contradictory for sufficiently small $\zeta$. Thus, necessarily $f_{\eps n} \leq 1+\ell$ for all large $n$ a.s., unless we have $e_{(1-\eps)n)} \geq f_{\eps n}$ in which case, as shown above, $f_{\eps n}\leq e_{(1-\eps)n)} \leq 1 + \ell$ for all large $n$ a.s.

Altogether, we necessarily have
\begin{align}
\max\{ e_{(1-\eps)n}, f_{\eps n} \} \leq 1+\ell
\notag
\end{align}
for all large $n$ a.s. All the same, by reverting the inequalities, we prove that, for all large $n$ a.s.
\begin{align}
\min\{ e_1, f_1 \} \geq 1 - \ell
\notag
\end{align}
and therefore, altogether,
\begin{align}
 \max_{1\leq i \leq (1-\epsilon)n} | e_i -1 | & \leq \ell \notag \\
 \max_{1\leq i \leq \epsilon n} | f_i -1 | & \leq \ell \notag 
\end{align}
for all large $n$ a.s., which eventually proves \eqref{max_ei} and \eqref{max_fi} by taking a countable sequence of $\ell$ going to zero. This establishes the main result, from which Theorem \ref{thm1} unfolds. Specifically, from \eqref{max_ei}-\eqref{max_fi} and by uniform boundedness of $\gammaEq$ and $\alphaEqi$,
\begin{align*}
 \max_{1\leq i \leq (1-\eps)n} | v(d_i) - v(\gammaEq) | & \toas 0 \notag \\
 \max_{1\leq i \leq \eps n} | v(b_i) - v(\alphaEqi) | & \toas 0 \notag  .
\end{align*}
Thus, for any $\ell >0$ and for all large $n$ a.s.
\begin{align*}
(1-\ell) \Shat \preceq \Chat \preceq (1+\ell) \Shat
\end{align*}
and, therefore $\| \Chat - \Shat \| \leq  2 \ell \| \Shat \|$. Using the triangle inequality and the fact that $v$ is non-increasing, we have
\begin{align*}
& \| \Chat - \Shat \|   \notag \\
& \leq 2 \ell \, v(0) \left( \left\| \frac{1}{n} \sum\nolimits_{i=1}^{(1-\varepsilon_n)n}  \x_i \x_i^\dagger  \right\|  + \left\| \frac{1}{n} \sum\nolimits_{i=1}^{\varepsilon_n n} \a_i \a_i^\dagger  \right\| \right) .
\end{align*}
From \cite{Bai1998} and Assumption \ref{as:cN}, $\| \frac{1}{n} \sum\nolimits_{i=1}^{(1-\varepsilon_n)n}  \x_i \x_i^\dagger  \| < 4 (1-\varepsilon)$ for all large $n$ a.s. and, from Assumption \ref{as:C}, $\limsup_n \| \frac{1}{n} \sum\nolimits_{i=1}^{\varepsilon_n n} \a_i \a_i^\dagger  \|  < \infty$. Then, since $\ell$ is arbitrarily small, $\| \Chat - \Shat \|$ tends to zero a.s. as $n \to \infty$, which concludes the proof of Theorem~\ref{thm1}. For $\C \neq \I$ we simply need to show $\| \Chalf (\Chat - \Shat) \Chalf \| \toas 0$, which follows from $ \| \Chalf (\Chat - \Shat) \Chalf \| \leq \| \C \| \| \Chat - \Shat \|$ since, by assumption, $\limsup_N \| \C \| < \infty$.


\bigskip

For the random outliers scenario, Assumption \ref{as:C} holds a.s. by virtue of \cite{Bai1998}, provided that $\limsup_N \| \D \C^{-1} \| < \infty$. Then, the proof of Corollary~\ref{cor1} follows from applying standard random matrix arguments to the model of $\Shat$ in Theorem~\ref{thm1}, considered now as a random matrix in both $\y_i$ and $\a_i$. The result may be straightforwardly obtained from, e.g., \cite[Thm.~1]{Wagner2012} (see Appendix~\ref{app:RandomMatrixResults} for similar applications).

\section{Random Matrix Results}
\label{app:RandomMatrixResults}

In this section we list several intermediary results needed in Appendix~\ref{app:det}.

\begin{lemma}
Let Assumptions \ref{as:C}-\ref{as:cN} hold. Define
\begin{align}
\An \triangleq  \frac{1}{n} \sum_{j=1}^{(1-\eps)n} v\left( \gammaEq  \right) \x_j \x_j^\dagger
 + \frac{1}{n} \sum_{j=1}^{\eps n} v\left( \alphaEqj  \right) \a_j \a_j^\dagger   
 \notag
\end{align}
and $\Ani = \An - \frac1n v(\gammaEq) \x_i \x_i^\dagger$, with $\gammaEq$ and $\alphaEqj$ given in Theorem \ref{thm1}. Then, as $n \to \infty$,
\begin{align}
\max_{1 \leq i \leq \eps n}  \left|  \frac{1}{N} \x_i^\dagger \AniI \x_i      - \gammaEq \right| \toas 0 .
\notag
\end{align}
\label{lemma1}
\end{lemma}
\begin{proof}
We first need to establish a result on $\lambda_1 (\Ani)$, for which we know that $\lambda_1 (\Ani) \geq \lambda_1 (  v(\gammaEq) \frac{1}{n} \sum_{j \neq i}  \x_j \x_j^\dagger )$. Then, \cite[Lemma~1]{Couillet2014a} along with Assumption~\ref{as:cN} and the boundedness of $\gammaEq$ show that there exists $\xi > 0$ such that, for all large $n$ a.s.,
\begin{align}
\min_{1 \leq i \leq (1-\eps) n} \lambda_1 \left( \Ani \right) > \xi .
\label{eq:lambda1}
\end{align} 
With this acquired, the outline of the proof is divided into two main steps. We first prove that $\max_{1 \leq i \leq (1-\eps) n}  |  \frac1N \x_i^\dagger \AniI \x_i - \frac1N \tr \AnI | \toas 0$ using quadratic form-close-to-the trace and rank-one perturbation arguments. Then, using \cite[Thm~1]{Wagner2012}, we show that $ \left| \frac1N \tr \AnI - \gammaEq \right| \toas 0$.

The triangle inequality allows us to write
\begin{align} 
&\left|  \frac{1}{N} \x_i^\dagger \AniI \x_i - \frac1N \tr \, \AnI \right| \leq \nonumber \\
&\left|  \frac{1}{N} \x_i^\dagger \AniI \x_i - \frac{1}{N} \tr \, \AniI  \right| + \left| \frac{1}{N} \tr \AniI  -  \frac{1}{N} \tr \AnI  \right| .
\label{triangle}
\end{align}
Let us bound the two terms on the right hand side of \eqref{triangle}. Denote by $\Ex$ the expectation with respect to $\x_i$ (i.e., conditionally on $\Ani$) and $\kappa_i \triangleq \ind_{\{\lambda_1 (\Ani) > \xi \}}$ with $\xi$ defined in \eqref{eq:lambda1}. For the first term, we can apply \cite[Lemma~B.26]{Bai2009} (since $\x_i$ is independent of $\kappa_i^{1/p} \AniI$), so that for $p \geq 2$,
\begin{align}
& \Ex \left[ \kappa_i \left|  \frac{1}{N} \x_i^\dagger \AniI \x_i - \frac{1}{N} \tr \AniI  \right|^p \right] \nonumber \\
& \leq \frac{\kappa_i K_p}{N^{p/2}} \left[     \left( \frac{\nu_4}{N} \tr \left( \AniI \right)^2 \right)^{p/2} + \frac{\nu_{2p}}{N^{p/2}} \tr  \mathbf{F}_{N,(i)}^{-p}  \right]
\label{tracelemma}
\end{align}
for some constant $K_p$ depending only on $p$, with $\nu_{\ell}$ any value such that ${\rm{E}} \left[ |x_{ij} |^{\ell} \right] \leq \nu_{\ell}$. Using $\frac{1}{N^k} \tr \mathbf{B}^k \leq \left( \frac{1}{N} \tr \mathbf{B} \right)^k$ for $\mathbf{B} \in \mathbb{C}^{N \times N}$ nonnegative definite and $k \geq 1$ leads to
\begin{align} \label{bound1}
& \Ex \left[ \kappa_i \left|  \frac{1}{N} \x_i^\dagger \AniI \x_i - \frac{1}{N} \tr \AniI  \right|^p \right] \nonumber \\
& \leq \frac{\kappa_i K_p}{N^{p/2}}   \left( \nu_4^{p/2} + \nu_{2p} \right) \left( \frac{1}{N} \tr \mathbf{F}_{N,(i)}^{-2} \right)^{p/2} \nonumber \\
& \leq  \frac{K_p}{\xi^p N^{p/2}}  \left( \nu_4^{p/2} + \frac{\nu_{2p}}{N^{p/2 -1}} \right)
\end{align}
where for the second inequality we have used $\tr \mathbf{B} \leq \| \mathbf{B} \|$ for $\mathbf{B} \in \mathbb{C}^{N \times N}$ nonnegative definite and the fact that $\kappa_i \| \AniI \| < \xi^{-1}$, which holds from the definition of $\kappa_i$. The bound (\ref{bound1}) being irrespective of $\Ani$, we can now take the expectation over $\Ani$ to obtain
\begin{align} \label{bound1_final} 
 {\rm{E}} \left[ \kappa_i \left|  \frac{1}{N} \x_i^\dagger \AniI \x_i - \frac{1}{N} \tr \AniI  \right|^p \right] 
 &=  \O\left(\frac{1}{N^{p/2}} \right) .
\end{align}  

For the second term in (\ref{triangle}), we can write $\Ani = ( \Ani - \frac{\xi}{2} \I ) + \frac{\xi}{2} \I$ with $\Ani - \frac{\xi}{2} \I \succ \0$ and we have from \cite[Lemma~2.6]{Silverstein1995a} (rank-one perturbation lemma)
\begin{align} \label{bound2_final}
 {\rm{E}} \left[ \kappa_i \left| \frac{1}{N} \tr \AniI  -  \frac{1}{N} \tr \AnI  \right|^p \right] 
 & \leq \frac{1}{N^p} \left( \frac{2}{\xi} \right)^p .
\end{align} 

From \eqref{triangle}, we can now use H\"older's inequality and the bounds \eqref{bound1_final}--\eqref{bound2_final} to obtain
\begin{align} \label{bound_final}
 {\rm{E}} \left[ \kappa_i \left|  \frac{1}{N} \x_i^\dagger \AniI \x_i - \frac{1}{N} \tr \AnI  \right|^p \right] 
 &= \O\left(\frac{1}{N^{p/2}} \right) . 
\end{align}
Then, we have that
\begin{align}
& \Pr \left[ \max_{1 \leq i \leq (1-\eps)n} \kappa_i^{1/p} \left|  \frac{1}{N} \x_i^\dagger \AniI \x_i   - \frac{1}{N} \tr \AnI \right| > \zeta  \right]  \notag \\
& \leq \sum_{i=1}^{(1-\eps)n} \Pr \left[ \kappa_i^{1/p} \left|  \frac{1}{N} \x_i^\dagger \AniI \x_i      - \frac{1}{N} \tr \AnI \right| > \zeta \right] \nonumber \\
& \leq  \frac{(1-\eps) n}{\zeta^p}  {\rm{E}} \left[ \kappa_i \left|  \frac{1}{N} \x_i^\dagger \AniI \x_i - \frac{1}{N} \tr \AnI \right|^p \right] \nonumber \\
&=  \O\left(\frac{1}{N^{p/2 - 1}} \right)
\notag
\end{align}
where we have used (in order) Boole's inequality, Markov's inequality, and (\ref{bound_final}). Recall from (\ref{tracelemma}) that the entries of $\x_i$ are required to have finite $2p$-th order moment and that, by our initial assumption, ${\rm{E}} [ | x_{ij} |^{8+\eta} ] < \infty$ for some $\eta > 0$. Then, taking $p > 4$, the Borel Cantelli lemma along with the fact that $\min_{1 \leq i \leq (1-\eps)n} \kappa_i \toas 1$ ensure
\begin{align}
 \max_{1 \leq i \leq (1-\eps)n}  \left|  \frac{1}{N} \x_i^\dagger \AniI \x_i      - \frac{1}{N} \tr \AnI \right| \toas 0 .
 \label{conv1}
\end{align}

It remains to show that $\gammaEq$ is a deterministic equivalent for $\frac1N \tr \AnI$. From \eqref{eq:lambda1} and the fact that any subtraction of a nonnegative definite matrix cannot increase the smallest eigenvalue, we have that $\lambda_1 ( \An ) > \xi$ for all large $n$ a.s. Then, we can write $\An = ( \An - \frac{\xi}{2} \I ) + \frac{\xi}{2} \I$ with $\liminf_n \lambda_1 ( \An - \frac{\xi}{2} \I ) > 0$ a.s.\@ and we are in position to apply \cite[Thm.~1]{Wagner2012} which ensures
\begin{align}
 \left|  \frac{1}{N} \tr \AnI - \frac1N \tr \left( \frac{(1-\varepsilon) v(\gammaEq)}{1+e_N} \I + \A \right)^{-1} \right| \toas 0  \notag
\end{align}
where $\A = \frac{1}{n} \sum_{j=1}^{\eps n} v\left( \alphaEqj  \right) \a_j \a_j^\dagger$ and $e_N$ is the unique positive solution to
\begin{align*}
e_N = \c v(\gammaEq) \frac1N \tr \left( \frac{(1-\varepsilon) v(\gammaEq)}{1+e_N} \I + \A \right)^{-1} . 
\end{align*}
According to the definition of $\gammaEq$, $e_N = \c v(\gammaEq) \gammaEq$ with $\gammaEq$ the solution to
\begin{align*}
\gammaEq =  \frac1N \tr \left( \frac{(1-\varepsilon) v(\gammaEq)}{1+\c v(\gammaEq) \gammaEq} \I + \A \right)^{-1}
\end{align*}
which has been proven to be unique. Altogether,
\begin{align}
 \left|  \frac{1}{N} \tr \AnI - \gammaEq \right| \toas 0 .
 \label{conv2}
\end{align}
Combining \eqref{conv1} and \eqref{conv2} concludes the proof.
\end{proof}

\begin{lemma}
Let Assumptions~\ref{as:C}-\ref{as:cN} hold and define
\begin{align*}
\Bni \triangleq  \frac{1}{n} \sum_{j = 1}^{(1-\eps)n}  v\left( \gammaEq  \right) \x_j \x_j^\dagger
 + \frac{1}{n} \sum_{j \neq i} v\left( \alphaEqj  \right) \a_j \a_j^\dagger 
\end{align*}
with $\gammaEq$ and $\alphaEqj$ defined as in Theorem~\ref{thm1}. Then, as $n \to \infty$,
\begin{align}
\max_{1 \leq i \leq \eps n}  \left|  \frac{1}{N} \a_i^\dagger \BniI \a_i      - \alphaEqi \right| \toas 0 .
\notag
\end{align}
\label{lemma2}
\end{lemma}
\begin{proof}
Since $\lambda_1 (\Bni) \geq \lambda_1 ( v( \gammaEq ) \frac{1}{n} \sum_{j=1}^{(1-\eps)n}  \x_j \x_j^\dagger )$, we can use \cite[Lemma~1]{Couillet2014a} along with Assumption~\ref{as:cN} and the uniform boundedness of $\gammaEq$ to show that there exists $\xi > 0$ such that, for all large $n$ a.s.
\begin{align}
\min_{1 \leq i \leq \eps n} \lambda_1 \left( \Bni \right) > \xi .
\notag
\end{align}
Denote $\kappa_i \triangleq \ind_{\{\lambda_1 (\Bni) > \xi \}}$.
Using similar derivations as for \cite[Lemma 3]{Hachem2013} adapted to the present model, we have
\begin{align} \label{bound_final_lemma2}
 {\rm{E}} \left[ \kappa_i \left|  \frac{1}{N} \a_i^\dagger \BniI \a_i - \alphaEqi  \right|^p \right] 
 & = \O\left(\frac{1}{N^{p/2}} \right) .
\end{align}
Then
\begin{align*}
&\Pr \left[ \max_{1 \leq i \leq \eps n} \kappa_i^{1/p} \left|  \frac{1}{N} \a_i^\dagger \BniI \a_i  - \alphaEqi \right| > \zeta  \right] \nonumber \\ 
&\leq \sum_{i=1}^{\eps n} \Pr \left[ \kappa_i^{1/p} \left|  \frac{1}{N} \a_i^\dagger \BniI \a_i      - \alphaEqi \right| > \zeta \right] \\
&\leq  \frac{\eps n}{\zeta^p}  {\rm{E}} \left[ \kappa_i \left|  \frac{1}{N} \a_i^\dagger \BniI \a_i - \alphaEqi \right|^p \right] \\
&=  \O\left(\frac{1}{N^{p/2-1}} \right)
\end{align*}
where we used (in order) Boole's inequality, Markov's inequality, and \eqref{bound_final_lemma2}. Taking $p > 4$, the Borel Cantelli lemma ensures
\begin{align*}
 \max_{1 \leq i \leq \eps n} \kappa_i^{1/p}  \left|  \frac{1}{N} \a_i^\dagger \BniI \a_i      - \alphaEqi \right| \toas 0
\end{align*}
which then proves Lemma~\ref{lemma2} using $\min_{1 \leq i \leq \eps n} \kappa_i \toas 1$.
\end{proof}



\section{Asymptotic moments}
\label{app:moments}

In this last appendix, we derive the moments of the deterministic equivalents studied in \cite{Wagner2012}. We provide in full the generic result, which may be used for independent purposes. We first recall \cite[Thm.~1]{Wagner2012}.

\begin{theorem}[Wagner et al., \cite{Wagner2012}]
Let $\Y \in \mathbb{C}^{N \times n}$ have independent
columns $\y_i = \H_i \x_i$, where $\x_i \in \mathbb{C}^{N_i}$ has i.i.d.\@ entries of zero mean, variance $1/n$, and $4 + \eta$ moment of order $\O(1/n^{2+\eta /2})$, and $\H_i \in \mathbb{C}^{N \times N_i}$ such that
$\R_i \triangleq \H_i \H_i^\dagger$ has uniformly bounded spectral norm over $n, N$. Let also $\A \in \mathbb{C}^{N \times N}$ be Hermitian non-negative and denote $\An = \Y \Y^\dagger + \A$. Then, as $N$, $N_1,\ldots, N_n$, and $n$ grow large with ratios $c_i = N_i / n$, and $c_0 = N/n$ satisfying $0 <
\liminf_n c_i \leq \limsup_n c_i < \infty$ for $0 \leq i \leq n$, we have
\begin{align}
\frac1n \tr \left( \An - z \I \right)^{-1} - m_N(z) \toas 0
\notag
\end{align}
with
\begin{align} \label{stieltjies_app}
 &  m_N(z) = \frac1n
\tr \left( \frac1n \sum_{i=1}^n \frac{1}{1+e_{N,i}(z)} \R_i + \A - z \I  \right)^{-1} 
\end{align}
where $e_{N,1}(z),\ldots, e_{N,n}(z)$ form the unique solution of
\begin{align}
e_{N,j}(z) = \frac1n \tr \R_j  \left( \frac1n \sum_{i=1}^n \frac{1}{1  +  e_{N,i}(z)} \R_i   +   \A - z \I   \right)^{-1} \notag
\end{align}
such that all $e_{N,j}(z)$ are Stieltjes transforms of a non-negative finite measure on $\mathbb{R}^+$.
\label{moments:thm1}
\end{theorem}

From Theorem \ref{moments:thm1}, the distribution function $F_N$ with Stieltjes transform $m_N(z)$ is a deterministic equivalent for the eigenvalue distribution of $\An$. We next describe the successive moments of the distribution function $F_N$. This generalizes the asymptotic moment results in \cite{Hoydis2011}, valid only for $\A = \0$.

\begin{theorem}
Let $F_N$ be the distribution function associated with the Stieltjes transform (\ref{stieltjies_app}),
and denote $M_{N,0}, M_{N,1}, \ldots$ the successive moments of $F_N$, i.e.,
$M_{N,p} \triangleq \int x^p dF_N $.
Then,
\begin{align*}
M_{N,p} = \frac{(-1)^p}{p!} \frac1N \tr \T_p
\end{align*}
with $\T_0, \T_1, \ldots$ defined recursively from
\begin{align*}
\hspace{-3mm} \T_{p+1} &= - \hspace{-0.5mm} \sum_{i=0}^p \T_{p-i} \A \T_i \hspace{-0.8mm} + \hspace{-0.8mm} \sum_{i=0}^p \sum_{j=0}^i \hspace{-0.5mm} \binom{p}{i} \hspace{-0.5mm} \binom{i}{j} \T_{p-i} \Q_{i \hspace{-0.1mm} - \hspace{-0.1mm} j \hspace{-0.1mm} + \hspace{-0.1mm} 1} \T_{\hspace{-0.2mm} j}  \\
\hspace{-3mm} \Q_{p+1} &= \frac{p+1}{n} \sum_{k=1}^n f_{k,p} \R_k  \\
\hspace{-3mm} f_{k,p+1} &= \sum_{i=0}^p \sum_{j=0}^i \binom{p}{i} \binom{i}{j} (p-i+1) f_{k,j} f_{k,i-j} \deltaf_{k,p-i}  \\
\deltaf_{k,p+1} &= \frac1n \tr \left[ \R_k \T_{p+1} \right]  
\end{align*}
and $\T_0 = \I$, $f_{k,0} = -1$, $\deltaf_{k,0} = \frac1n \tr \R_k$ for $k \in \{ 1,\ldots, n \}$.
\label{moments:thm2}
\end{theorem}
\begin{proof}
Follows the same steps as the proof of \cite[Thm.~2]{Hoydis2011} with proper modifications to account for $\A \neq \0$. 
\end{proof}

\footnotesize

\bibliographystyle{IEEEtran}
\bibliography{IEEEabrv,refs}

\begin{thebibliography}{10}
\providecommand{\url}[1]{#1}
\csname url@samestyle\endcsname
\providecommand{\newblock}{\relax}
\providecommand{\bibinfo}[2]{#2}
\providecommand{\BIBentrySTDinterwordspacing}{\spaceskip=0pt\relax}
\providecommand{\BIBentryALTinterwordstretchfactor}{4}
\providecommand{\BIBentryALTinterwordspacing}{\spaceskip=\fontdimen2\font plus
\BIBentryALTinterwordstretchfactor\fontdimen3\font minus
  \fontdimen4\font\relax}
\providecommand{\BIBforeignlanguage}[2]{{%
\expandafter\ifx\csname l@#1\endcsname\relax
\typeout{** WARNING: IEEEtran.bst: No hyphenation pattern has been}%
\typeout{** loaded for the language `#1'. Using the pattern for}%
\typeout{** the default language instead.}%
\else
\language=\csname l@#1\endcsname
\fi
#2}}
\providecommand{\BIBdecl}{\relax}
\BIBdecl

\bibitem{Bianchi2009}
P.~Bianchi, J.~Najim, M.~Maida, and M.~Debbah, ``Performance analysis of some
  eigen-based hypothesis tests for collaborative sensing,'' in \emph{IEEE Stat.
  Signal Process. (SSP '09)}, Cardiff (UK), Sep. 2009, pp. 5--8.

\bibitem{Nadler2010}
B.~Nadler, ``Nonparametric detection of signals by information theoretic
  criteria: {Performance} analysis and an improved estimator,'' \emph{{IEEE}
  Trans. Signal Process.}, vol.~58, no.~5, pp. 2746--2756, 2010.

\bibitem{Mestre2008}
X.~Mestre and M.~A. Lagunas, ``Modified subspace algorithms for {DoA}
  estimation with large arrays,'' \emph{{IEEE} Trans. Signal Process.},
  vol.~56, no.~2, pp. 598--614, 2008.

\bibitem{BIC08}
P.~J. Bickel and E.~Levina, ``Regularized estimation of large covariance
  matrices,'' \emph{Ann. Stat.}, vol.~36, no.~1, pp. 199--227, 2008.

\bibitem{ELK10}
N.~{El Karoui}, ``The spectrum of kernel random matrices,'' \emph{Ann. Stat.},
  vol.~38, no.~1, pp. 1--50, 2010.

\bibitem{Huber1964}
P.~J. Huber, ``Robust estimation of a location parameter,'' \emph{Ann. Math.
  Stat.}, vol.~35, no.~1, pp. 73--101, 1964.

\bibitem{Maronna1976}
\BIBentryALTinterwordspacing
R.~A. Maronna, ``Robust {M}-estimators of multivariate location and scatter,''
  \emph{Ann. Stat.}, vol.~4, no.~1, pp. 51--67, 1976. [Online]. Available:
  \url{http://dx.doi.org/10.1214/aos/1176343347}
\BIBentrySTDinterwordspacing

\bibitem{Tyler1987}
D.~E. Tyler, ``A distribution-free {M-estimator} of multivariate scatter,''
  \emph{Ann. Stat.}, pp. 234--251, 1987.

\bibitem{Couillet2013}
R.~Couillet, F.~Pascal, and J.~W. Silverstein, ``The random matrix regime of
  {Maronna's} {M}-estimator with elliptically distributed samples,''
  \emph{arXiv preprint arXiv:1311.7034}, 2013.

\bibitem{COU14}
R.~Couillet and M.~McKay, ``Large dimensional analysis and optimization of
  robust shrinkage covariance matrix estimators,'' \emph{J. Multivar. Anal.},
  vol. 131, pp. 99--120, 2014.

\bibitem{COU14d}
\BIBentryALTinterwordspacing
R.~Couillet, A.~Kammoun, and F.~Pascal, ``Second order statistics of robust
  estimators of scatter. {Application} to {GLRT} detection for elliptical
  signals,'' \emph{(submitted to) J. Multivar. Anal.}, 2014. [Online].
  Available: \url{http://arxiv.org/abs/1410.0817}
\BIBentrySTDinterwordspacing

\bibitem{ZHA14}
T.~Zhang, X.~Cheng, and A.~Singer, ``{Marchenko-Pastur} law for {Tyler's} and
  {Maronna's} {M}-estimators,'' \emph{http://arxiv.org/abs/1401.3424}, 2014.

\bibitem{YAN14}
L.~Yang, R.~Couillet, and M.~McKay, ``Minimum variance portfolio optimization
  with robust shrinkage covariance estimation,'' in \emph{IEEE Asilomar Conf.
  Sig. Sys. Comput.}, Pacific Grove, CA, USA, Nov. 2014.

\bibitem{COU14c}
\BIBentryALTinterwordspacing
R.~Couillet, ``Robust spiked random matrices and a robust {G-MUSIC}
  estimator,'' \emph{(submitted to) J. Multivar. Anal.}, 2014. [Online].
  Available: \url{http://arxiv.org/pdf/1404.7685}
\BIBentrySTDinterwordspacing

\bibitem{CHI14}
Y.~Chitour, R.~Couillet, and F.~Pascal, ``On the convergence of maronna's
  m-estimators of scatter,'' \emph{{IEEE} Signal Process. Lett.}, vol.~22,
  no.~6, pp. 709--712, 2014.

\bibitem{Kent1991}
J.~T. Kent and D.~E. Tyler, ``Redescending {M}-estimates of multivariate
  location and scatter,'' \emph{Ann. Stat.}, pp. 2102--2119, 1991.

\bibitem{Horn1985}
R.~A. Horn and C.~R. Johnson, \emph{Matrix Analysis}.\hskip 1em plus 0.5em
  minus 0.4em\relax Cambridge University Press, 1985.

\bibitem{Couillet2014a}
R.~Couillet, F.~Pascal, and J.~Silverstein, ``Robust estimates of covariance
  matrices in the large dimensional regime,'' \emph{{IEEE} Trans. Inf. Theory},
  vol.~60, no.~11, pp. 7269--7278, 2014.

\bibitem{Silverstein1995a}
J.~W. Silverstein and Z.~Bai, ``On the empirical distribution of eigenvalues of
  a class of large dimensional random matrices,'' \emph{J. Multivar. Anal.},
  vol.~54, no.~2, pp. 175--192, 1995.

\bibitem{POT00}
L.~Laloux, P.~Cizeau, M.~Potters, and J.~P. Bouchaud, ``{Random matrix theory
  and financial correlations},'' \emph{Int. J. Theoretical Appl. Finance},
  vol.~3, no.~3, pp. 391--397, Jul. 2000.

\bibitem{QUA13}
A.~A. Quadeer, R.~H.~Y. Louie, K.~Shekhar, A.~K. Chakraborty, I.-M. Hsing, and
  M.~R. McKay, ``{Statistical linkage of mutations in the non-structural
  proteins of hepatitis C virus exposes targets for immunogen design},''
  \emph{J. Virology}, vol.~88, no.~13, pp. 7628--7644, 2014.

\bibitem{CHE11}
Y.~Chen, A.~Wiesel, and A.~O. Hero, ``{Robust shrinkage estimation of
  high-dimensional covariance matrices},'' \emph{{IEEE} Trans. Signal
  Process.}, vol.~59, no.~9, pp. 4097--4107, 2011.

\bibitem{LED04}
O.~Ledoit and M.~Wolf, ``A well-conditioned estimator for large-dimensional
  covariance matrices,'' \emph{J. Multivar. Anal.}, vol.~88, no.~2, pp.
  365--411, 2004.

\bibitem{Yates1995}
R.~Yates, ``A framework for uplink power control in cellular radio systems,''
  \emph{{IEEE} J. Sel. Areas Commun.}, vol.~13, no.~7, pp. 1341--1347, 1995.

\bibitem{Wagner2012}
S.~Wagner, R.~Couillet, M.~Debbah, and D.~T. Slock, ``Large system analysis of
  linear precoding in correlated {MISO} broadcast channels under limited
  feedback,'' \emph{{IEEE} Trans. Inf. Theory}, vol.~58, no.~7, pp. 4509--4537,
  2012.

\bibitem{Bai1998}
Z.~Bai and J.~W. Silverstein, ``No eigenvalues outside the support of the
  limiting spectral distribution of large-dimensional sample covariance
  matrices,'' \emph{Ann. Prob.}, pp. 316--345, 1998.

\bibitem{Bai2009}
Z.~D. Bai and J.~W. Silverstein, \emph{Spectral Analysis of Large Dimensional
  Random Matrices}, 2nd~ed.\hskip 1em plus 0.5em minus 0.4em\relax New York,
  NY, USA: Springer Series in Statistics, 2009.

\bibitem{Hachem2013}
W.~Hachem, P.~Loubaton, X.~Mestre, J.~Najim, and P.~Vallet, ``A subspace
  estimator for fixed rank perturbations of large random matrices,'' \emph{J.
  Multivar. Anal.}, vol. 114, pp. 427--447, 2013.

\bibitem{Hoydis2011}
J.~Hoydis, M.~Debbah, and M.~Kobayashi, ``Asymptotic moments for interference
  mitigation in correlated fading channels,'' in \emph{IEEE Int. Symp. Inf.
  Theory (ISIT)}, St. Petersburg (Russia), Jul. 2011, pp. 2796--2800.

\end{thebibliography}

\end{document}